\documentclass[11pt, reqno]{amsart}

\usepackage{etoolbox}

\newbool{ForSubmission}
\booltrue{ForSubmission}

\usepackage{amscd,amsmath}
\usepackage{mathrsfs}
\usepackage{amsfonts}
\usepackage{amssymb, bm, xspace}
\usepackage{enumerate}
\usepackage{setspace}

\usepackage{datetime}

\usepackage{enumitem}

\usepackage{xcolor}

\usepackage[pagebackref=true, colorlinks=true, citecolor=blue]{hyperref}

\usepackage[margin=1.2in]{geometry}

\usepackage{mathtools}

\usepackage{cleveref}

\usepackage{stackengine}
\stackMath

\newcommand{\Comment}[1]{{\color{blue}#1}}
\newcommand{\OptionalDetails}[1]{
    \ifbool{ForSubmission}
        {
        }
        {\begin{quote}\Comment{\footnotesize
        \medskip

        \noindent#1}
        \end{quote}
        }
    }

\newbool{HaveBBM}
\booltrue{HaveBBM}

\ifbool{HaveBBM}{
	\usepackage{bbm}
    }
    {
    }

\newbool{arXivFormat}
\boolfalse{arXivFormat}
    
\newcommand{\IfarXivElse}[2]{
    \ifbool{arXivFormat}
        {#1}{#2}
    }

\renewcommand{\mathbf}[1]{\bm{#1} \textbf{ *** Use bm instead of mathbf ***}}

\newcommand{\eqn}{\begin{eqnarray}}
\newcommand{\een}{\end{eqnarray}}

\newtheorem{theorem}{Theorem}[section]
\newtheorem*{theorem*}{Theorem}				
\newtheorem{prop}[theorem]{Proposition}
\newtheorem{lemma}[theorem]{Lemma}
\newtheorem{cor}[theorem]{Corollary}

\newtheorem{definition}[theorem]{Definition}
\newtheorem{remark}[theorem]{Remark}

\numberwithin{equation}{section}

\newcommand{\orgabs}[1]{\left\vert#1\right\vert}
\DeclarePairedDelimiter{\abs}{\lvert}{\rvert}
\DeclarePairedDelimiter{\bigabs}{\Big\lvert}{\Big\rvert}
\DeclarePairedDelimiter{\Bigabs}{\Bigg\lvert}{\Bigg\rvert}

\newcommand{\innp}[1]{\ensuremath{\langle #1 \rangle}}

\newcommand{\BoldTau}
    {\mbox{\boldmath $\tau$}}

\newcommand{\BB}[1]{\ensuremath{\mathbb{#1}}}

\newcommand{\R}{\ensuremath{\BB{R}}}
\newcommand{\N}{\ensuremath{\BB{N}}}
\newcommand{\Z}{\ensuremath{\BB{Z}}}
\newcommand{\iny}{\ensuremath{\infty}}
\newcommand{\grad}{\ensuremath{\nabla}}

\newcommand{\CharFunc}{
    \ifbool{HaveBBM}{
        \ensuremath{\mathbbm{1}}
        }
        {
        \ensuremath{\bm{1}}
        }
    }

\DeclareMathOperator{\dv}{div} %
\DeclareMathOperator{\curl}{curl} %
\DeclareMathOperator{\supp}{supp} %
\newcommand{\prt}{\ensuremath{\partial}}
\newcommand{\brac}[1]{\ensuremath{\left[ #1 \right]}}

\newcommand{\pr}[1]{\ensuremath{\left( #1 \right)}}

\DeclarePairedDelimiter{\set}{\{}{\}}

\newcommand{\norm}[1]{\ensuremath{\left\Vert #1 \right\Vert}}

\newcommand{\smallnorm}[1]{\ensuremath{\Vert #1 \Vert}}

\newcommand{\AGE}{aggregation equation\xspace}

\newcommand{\GAGnu}{\ensuremath{(GAG_\nu)}\xspace}
\newcommand{\GAGzero}{\ensuremath{(GAG_0)}\xspace}

\newcommand{\PAGnu}{\ensuremath{(AG_\nu)}\xspace}
\newcommand{\PAGzero}{\ensuremath{(AG_0)}\xspace}

\newcommand{\GAGVnu}{\ensuremath{(GAGV_\nu)}\xspace}

\newcommand{\VV}{\ensuremath{(VV)}\xspace}

\newcommand{\btau}{\BoldTau}
\newcommand{\mul}{\widetilde{\mu}}
\newcommand{\wl}{\widetilde{\w}}
\newcommand{\rhol}{\widetilde{\rho}}
\newcommand{\vl}{\widetilde{\v}}

\newcommand\tenq[2][1]{%
	\def\useanchorwidth{T}%
	\ifnum#1>1%
		\stackunder[0pt]{\tenq[\numexpr#1-1\relax]{#2}}{\scriptscriptstyle\sim}%
	\else%
		\stackunder[1pt]{#2}{\scriptscriptstyle\sim}%
	\fi%
	}

\newcommand{\FTF}
    {\Cal{F}}

\newcommand{\F}{\Phi}

\renewcommand{\epsilon}{\varepsilon}
\newcommand{\eps}{\ensuremath{\varepsilon}}

\newcommand{\Cal}[1]{\ensuremath{\mathcal{#1}}}
\newcommand{\al}{\ensuremath{\alpha}}
\newcommand{\la}{\ensuremath{\lambda}}

\newcommand{\diff}[2]{\frac{ d#1}{d#2}}

\newcommand{\MOC}{modulus of continuity\xspace}

\newcommand{\Holder}
    {H\"{o}lder\xspace}

\newcommand{\Ignore}[1]{}

\newcommand{\ToDo}[1]{\textbf{\Comment{[#1]}}}

\definecolor{Correction}{named}{red}

\crefname{cor}{Corollary}{Corollaries} 
									   
\crefname{lemma}{Lemma}{Lemmas}	       

\crefname{section}{Section}{Sections}
\Crefname{section}{Section}{Sections}

\crefname{appendix}{Appendix}{Appendices}
\Crefname{appendix}{Appendix}{Appendices}

\crefname{theorem}{Theorem}{Theorems}
\Crefname{theorem}{Theorem}{Theorems}

\crefname{prop}{Proposition}{Propositions}
\Crefname{prop}{Proposition}{Propositions}

\crefname{conj}{Conjecture}{Conjectures}
\Crefname{conj}{Conjecture}{Conjectures}

\crefname{definition}{Definition}{Definitions}
\Crefname{definition}{Definition}{Definitions}

\crefname{remark}{Remark}{Remarks}
\Crefname{remark}{Remark}{Remarks}

\crefname{assumption}{Assumption}{Assumptions}
\Crefname{assumption}{Assumption}{Assumptions}

\crefformat{equation}{(#2#1#3)}
\crefrangeformat{equation}{(#3#1#4) through (#5#2#6)}
\crefmultiformat{equation}
    {(#2#1#3)}%
    { and~(#2#1#3)}
    {, (#2#1#3)}
    { and~(#2#1#3)}

\renewcommand{\i}{\bm{\mathrm{i}}}

\renewcommand{\u}{\bm{\mathrm{u}}}
\renewcommand{\v}{\bm{\mathrm{v}}}
\newcommand{\w}{\bm{\mathrm{w}}}

\newcommand{\E}{\ensuremath{\mathbb{E}}}

\begin{document}

\title
	[The aggregation equation with Newtonian potential]
	{The aggregation equation with Newtonian potential}

\author[Elaine Cozzi, Gung-Min Gie, James P. Kelliher]
{Elaine Cozzi$^1$, Gung-Min Gie$^2$, James P. Kelliher$^3$}
\address{$^1$Department of Mathematics, 368 Kidder Hall, Oregon State University, Corvallis, OR 97330, U.S.A.}
\address{$^2$ Department of Mathematics, 328 Natural Sciences Building, University of Louisville, Louisville, KY 40292, U.S.A.}
\address{$^3$ Department of Mathematics, University of California, Riverside, 900 University Ave., Riverside, CA 92521, U.S.A.}
\email{cozzie@math.oregonstate.edu}
\email{gungmin.gie@louisville.edu}
\email{kelliher@math.ucr.edu}

\begin{abstract}
	The viscous and inviscid \AGE with Newtonian potential models a number
	of different physical systems, and has close analogs in 2D incompressible
	fluid mechanics. We consider a slight generalization of these equations
	in the whole space, establishing well-posedness and spatial decay
	of the viscous equations, and obtaining the convergence of viscous
	solutions to the inviscid solution as the viscosity goes to zero.
\end{abstract}

\maketitle

	\begin{center}
		

	\end{center}

	\tableofcontents
	

%
%
\section{Introduction}\label{S:Introduction}

\noindent In this work, we study, on $\R^d$, $d \ge 2$, the (viscous or inviscid) \AGE with Newtonian potential,
\begin{align*} 
	\PAGnu \quad
	\left\{
	\begin{array}{l}
		\prt_t \rho^\nu + \dv (\rho^\nu \v^\nu) = \nu \Delta \rho^\nu, \\
		\v^\nu = - \grad \F * \rho^\nu, \\
		\rho^\nu(0) = \rho_0.
	\end{array}
	\right.
\end{align*}
Here, $\nu \ge 0$ is the viscosity and $\F$ is the fundamental solution of the Laplacian, or Newtonian potential (so $\Delta \F = \delta$ and $\dv \v^\nu = - \rho^\nu$). 
The density is $\rho^\nu$, the velocity is $\v^\nu$, and $\rho_0$ is the initial density. 

Many variations on these equations are considered in the literature, primarily by using potentials other than the Newtonian or by using more general diffusive terms. We restrict our attention to the Newtonian potential with linear diffusion, for we will be concerned with analyzing the viscous ($\nu > 0$) and inviscid ($\nu = 0$) \AGE using techniques adapted from the study of 2D fluid mechanics.

The \AGE models many different physical problems.
For the Newtonian potential, as in \PAGnu, this includes type-II superconductivity when $\nu = 0$ (see \cite{MasmoudiZhang2005} and the references therein) and chemotaxis, where \PAGnu for $\nu > 0$ is a limiting case of the Keller-Segel system (see Section 5.2 of \cite{Perthame2007}), and has been extensively studied. In this context, $\rho^\nu$ measures the density of cells (bacteria or cancer cells, for instance) and $\v^\nu$ is the gradient of the concentration of a chemoattractant. References most closely related to the approach to the \AGE taken in this paper include \cite{BGLV2015,BLL2012,BlanchetCarrilloMasmoudi2008,KiselevXu2015,Laurent2007a,MasmoudiZhang2005,Perthame2007}.

We will, in fact, consider a slightly more general set of equations of the form
\begin{align*}
	\GAGnu \quad
	\left\{
	\begin{array}{l}
		\prt_t \rho^\nu + \v^\nu \cdot \grad \rho^\nu
			= \sigma_2 (\rho^\nu)^2 + \nu \Delta \rho^\nu, \\
		\v^\nu = \sigma_1 \grad \F * \rho^\nu, \\
		\rho^\nu(0) = \rho_0,
	\end{array}
	\right.
\end{align*}
where $\sigma_1$, $\sigma_2$ are constants with $\sigma_1 \ne 0$.
When $\sigma_1 = -1$, $\sigma_2 = 1$, \GAGnu reduces to \PAGnu, since then $\dv(\rho^\nu \v^\nu) = \v^\nu \cdot \grad \rho^\nu + \dv \v^\nu \rho^\nu = \v^\nu \cdot \grad \rho^\nu - (\rho^\nu)^2$.
We will study these equations in all of $\R^d$, though much of what we find extends naturally to a bounded domain given appropriate boundary conditions.

At least one other special case of \GAGnu has been studied in the literature: \GAGzero with $\sigma_1 = -1$, $\sigma_2 = 0$ are derived from \PAGzero by making a transformation of variables in (1.6) of \cite{BGLV2015}. This transformation applies only in the special case of aggregation patch initial data (analogous to vortex patches for fluids) for \PAGzero. Although this transformation only works for aggregation patch initial data, the authors of \cite{BGLV2015} go on to use this special case of \GAGzero throughout their analysis of aggregation patches. A general well-posedness result is not needed in \cite{BGLV2015} and hence not established there, but such a result was one of our motivations for studying the generalization of \PAGnu in \GAGnu, the parameters $\sigma_1$, $\sigma_2$ merely interpolating between \PAGzero and the equations studied in \cite{BGLV2015}.

We will find establishing the existence of weak viscous solutions to \GAGnu no more difficult than doing the same for \PAGnu, except for keeping track of the constants $\sigma_1$ and $\sigma_2$. We give a proof of existence of weak solutions for $\rho_0 \in L^1 \cap L^\iny$ in \cref{S:ViscousProblem}. The result we obtain, specifically a bound on the existence time, is suited to our needs in later sections, though much more is known about the existence time of solutions, at least for \PAGnu for nonnegative $\rho_0$ (as summarized in Sections 5.2, 5.3 of \cite{Perthame2007}). Uniqueness for solutions to \GAGnu when $\sigma_1 + \sigma_2 = 0$ follows, even for $\rho_0 \in BMO$, using Yudovich's uniqueness argument (in the form in \cite{Y1995}) as proved in \cite{BedrossianAzzam2015}.

In \cref{S:SpatialDecay} we bound the spatial decay of viscous solutions, bounds that will be required later in establishing the vanishing viscosity limit.

The varying effects of $\sigma_1$ and $\sigma_2$ begin to become apparent when we examine the behavior of the total mass of the density, $m(\rho^\nu) := \int_{\R^d} \rho^\nu$, in \cref{S:InfiniteEnergy}. We will find that $m(\rho^\nu)$ is conserved only when $\sigma_1 + \sigma_2 = 0$.

The well-posedness of the inviscid equations, \GAGzero, are the subject of \cref{S:InviscidProblem}.

In \cref{S:VV} we begin our analysis of the vanishing viscosity limit of solutions of $\GAGnu$ to a solution to $\GAGzero$ with the same initial data, showing that
\begin{align*}
	\VV:
		\quad
			\v^\nu \to \v^0 \text{ in } L^\iny(0, T; H^1), \; \rho^\nu \to \rho^0
				\text{ in } L^\iny(0, T; L^2) \text{ as } \nu \to 0.
\end{align*}
When $d = 3$, $\v^\nu$ and $\v^0$ both lie in $L^2(\R^d)$. When $d = 2$, this is no longer (in general) the case, the energies being infinite. When $\sigma_1 + \sigma_2 = 0$, however, because the total mass of the densities $\rho^{\nu}$ and $\rho^0$ are conserved over time, the infinite parts of the energies cancel, giving $\v^\nu - \v^0 \in L^2(\R^2)$. In both of the cases, $d \ge 3$ or $d = 2$ with $\sigma_1 + \sigma_2 = 0$, \VV holds, as we show in \cref{S:VV}.

In \cref{S:VVNonL2} we consider the remaining case where $d = 2$ but $\sigma_1 + \sigma_2 \ne 0$. In this case, the total mass of the densities are not conserved over time, and the infinite parts of the energies do not cancel. We will nonetheless be able to isolate the infinite parts of the energy and use them to define a corrector, $\bm\theta^\nu$, that lies in weak-$L^2$ and all higher $L^p$ spaces, and show that in place of \VV, we have
\begin{align*}
	\VV':
		&\quad
			\v^\nu - \v^0 - \bm\theta^\nu \text{ in } L^\iny(0, T; H^1), \; \rho^\nu \to \rho^0
				\text{ in } L^\iny(0, T; L^2) \text{ as } \nu \to 0, \\
		&\quad
			\bm\theta^\nu \to 0 \text{ in } L^\iny([0, T]; C^k)
				\text{ for all } k \ge 0.
\end{align*}

As can be seen from \VV, $\VV'$, both the velocity and density converge strongly in the vanishing viscosity limit. Indeed, the arguments in \cref{S:VV,S:VVNonL2} involve showing the simultaneous convergence of both the velocities and the densities.

In \cref{S:VVUniform}, we use the results from \cref{S:VV,S:VVNonL2}, along with uniform bounds in viscosity on Holder norms of solutions to ($GAG_{\nu}$), to prove that the vanishing viscosity limit holds in the $L^{\infty}$-norm of the density.

\bigskip

We close by stating a few conventions and making one definition.

We follow the convention that $\norm{\cdot} = \norm{\cdot}_{L^2(\R^d)}$.
We write $\innp{\cdot, \cdot}$ for the $L^2$-inner product and $(\cdot, \cdot)$ for the pairing in the duality between $H^1$ and $H^{-1}$.

Our proofs of the vanishing viscosity limit yield rates of convergence in both viscosity and time, but it will be unwieldy to keep track of the specific dependence on time. Hence, we will use the convention that
\begin{align*} 
	C_0(t) \text{ is a positive, continuous, nondecreasing function of } t \in [0, \iny).
\end{align*}

We make the convention that $C(a_1, \dots, a_n)$ stands for a continuous function from $[0, \iny)^n \to [0, \iny)$ that is nondecreasing in each of its arguments. We use $C(a_1, \dots, a_n)$ in the context of a constant that depends on the parameters $a_1, \dots, a_n$, where the exact form of the constant is unimportant.

We will find various uses for the following cutoff function:
\begin{definition}\label{D:Radial}
	Let $a$ be a radially symmetric function in $C^\iny(\R^d)$ supported in $B_2(0)$ with
	$a \equiv 1$ on $B_1(0)$ and with $a(x)$ nonincreasing in $\abs{x}$.
	For any $R \ge 1$ define $a_R(\cdot) = a(\cdot/R)$. Note
	that for any fixed $x \in \R^d$, $a_R(x)$ is nondecreasing in $R$.
\end{definition}

For any $p_1, p_2 \in [1, \iny]$, we define $\norm{f}_{L^{p_1} \cap L^{p_2}} = \norm{f}_{L^{p_1}} + \norm{f}_{L^{p_2}}$.

%
%
\section{The viscous problem}\label{S:ViscousProblem}

\noindent \cref{D:WeakSolution} gives our definition of a weak solution to the \AGE. This definition applies for both viscous and inviscid solutions.

\begin{definition}\label{D:WeakSolution}
	Let $\nu \ge 0$ and $\rho_0 \in L^1 \cap L^\iny$.
	We say that $\rho^\nu$ is a weak solution to the aggregations equations \GAGnu
	on the interval $[0, T]$
	with initial density $\rho_0$ if $\rho^\nu(0) = \rho_0$,
	\begin{align}\label{e:MassReg}
		\begin{split}
		&\rho^\nu \in L^\iny(0, T; L^1 \cap L^\iny)
			\cap C([0, T]; L^2), \\
		&\prt_t \rho^\nu \in L^2(0, T; H^{-1}), \\
		&\rho^\nu \in L^2(0, T; H^1) \quad \text{if } \nu > 0, \\
		\end{split}
	\end{align}
	with
	\Ignore{ 
	\begin{align}\label{e:WeakDefEq}
		\innp{\prt_t \rho^\nu, \varphi} - \innp{\rho^\nu \v^\nu, \grad \varphi}
			&= (\sigma_1 + \sigma_2) \innp{(\rho^\nu)^2, \varphi}
				- \nu \innp{\grad \rho^\nu, \grad \varphi}
	\end{align}
	for all $\varphi \in H^1(\R^d)$, where $\v^\nu
	:= \sigma_1 \grad \F * \rho^\nu$.
	} 
	\begin{align}\label{e:WeakDefEq}
		\begin{split}
		&\int_0^T \int_{\R^d}
			\pr{\rho^\nu \prt_t \varphi
				+ \rho^\nu \v^\nu \cdot \grad \varphi
				+ (\sigma_1 + \sigma_2) (\rho^\nu)^2 \varphi
				- \nu \grad \rho^\nu \cdot \grad \varphi
				}
			= 0 \\
		&\qquad\qquad\qquad
		\text{for all }\varphi \in C_0^\iny([0, T) \times \R^d).
		\end{split}
	\end{align}
\end{definition}
\begin{remark}
	By the initial condition $\rho^\nu(0) = \rho_0$ we mean that
	$\rho^\nu(t) \to \rho_0$ in $L^2$ as $t \to 0^+$, which makes sense
	because $\rho^\nu \in C([0, T]; L^2)$.
\end{remark}



In this section we treat weak solutions to \GAGnu for $\nu > 0$. In \cref{S:InviscidProblem}
we treat the case $\nu = 0$.

Define the total mass of $f \in L^1(\R^d)$ by
\begin{align}\label{e:MassDef}
	m(f)
		:= \int_{\R^d} f.
\end{align}

We consider first in \cref{S:LinearProblem} a higher regularity linear problem that we will use in \cref{S:NonlinearProblem} to obtain a solution to the nonlinear problem (that is, a solution as in \cref{D:WeakSolution}).

\subsection{The linear problem}\label{S:LinearProblem}

\begin{prop}\label{P:LinearProblem}
Let $\nu > 0$. For a fixed $T > 0$ let $f \in L^\iny(0, T; L^1 \cap L^\iny) \cap C(0, T; L^2)$ and let $\v_f = \sigma_1 \grad \F * f$. Then there exists a unique weak solution $\rho^\nu \in C(0, T; L^1 \cap L^\iny) \cap L^2(0, T; H^1)$ to
\begin{align}\label{e:AGLinear}
	\left\{
	\begin{array}{l}
		\prt_t \rho^\nu + \v_f \cdot \grad \rho^\nu = \sigma_2 f \rho^\nu + \nu \Delta \rho^\nu, \\
		\rho^\nu(0) = f(0).
	\end{array}
	\right.
\end{align}
Moreover, if $f$ also lies in $L^\iny(0, T;  C^\iny)$ then $\rho^\nu$ also lies in $L^\iny(0, T; C^\iny)$ and is unique.
\end{prop}
\begin{proof}
	We can write \cref{e:AGLinear} in weak form as
	\begin{align*}
		(\prt_t \rho^\nu, \varphi) + a(\rho^\nu, \varphi) = 0
	\end{align*}
	a.e. in time for all $\varphi \in H^1(\R^d)$, where the bilinear form,
	$a \colon H^1(\R^d) \times H^1(\R^d) \to \R$, is given by
	\begin{align*}
		a(g, \varphi)
			:= \innp{\v_f \cdot \grad g - \sigma_2 f g, \varphi}
				+ \nu \innp{\grad g, \grad \varphi}.
	\end{align*}
	
	Observe that
	\begin{align*}
		\abs{a(g, \varphi)}
			&\le C(f) \norm{g}_{H^1(\R^d)} \norm{\varphi}_{H^1(\R^d)}
	\end{align*}
	and
	\begin{align*}
		&a(g, g)
			= \innp{\v_f \cdot \grad g - \sigma_2 f g, g}
				+ \nu \norm{\grad g}_{L^2(\R^d)}^2
			= \nu \norm{\grad g}_{L^2(\R^d)}^2
				+ \frac{1}{2} \int_{\R^d} \v_f \cdot \grad g^2
				- \sigma_2 \int_{\R^d} f g^2 \\
			&= \nu \norm{\grad g}_{L^2(\R^d)}^2
				- \pr{\frac{\sigma_1}{2} + \sigma_2} \int_{\R^d} f g^2 
			\ge
				\nu \norm{\grad g}_{L^2(\R^d)}^2
					- C(f) \norm{g}_{L^2(\R^d)}^2. 
	\end{align*}
	The existence of a unique weak solution to \cref{e:AGLinear} with
	$\rho^\nu \in C(0, T; L^1 \cap L^\iny) \cap L^2(0, T; H^1)$ then
	follows from \cite{LionsMagenes1972} (see Theorem 10.9 of \cite{Brezis2011}).
	That $f$ also in $L^\iny(0, T; C^\iny)$ gives $\rho^\nu$ in $L^\iny(0, T; C^\iny)$
	follows via a standard bootstrap argument.
\end{proof}

\Ignore { 
To obtain a smooth solution, we follows the usual procedure of finding a weak solution then proving higher regularity. Since the approach to higher regularity is very similar to that taken in \cref{S:ViscousHigherRegularity}, we give only the proof of the existence of a weak solution.

We define a weak solution as in \cref{D:WeakSolution} but with \cref{e:WeakDefEq} replaced by
\begin{align}\label{e:LinearWeakDefEq}
	\innp{\prt_t \rho, \varphi} - \innp{\rho \v, \grad \varphi}
		&= (\sigma_1 + \sigma_2) \innp{\rho^2, \varphi}
			- \nu \innp{\grad \rho, \grad \varphi}
\end{align}
for all $\varphi \in H^1(\R^d)$.

Let $(g_n)_{n = 1}^\iny$ be
} 

\subsection{The nonlinear problem}\label{S:NonlinearProblem}

\begin{theorem}\label{T:ViscousExistence}
Fix $T > 0$ with $T < (\abs{\sigma_2} \norm{\rho_0}_{L^\iny})^{-1}$ or
	$T < \iny$ if $\sigma_2 = 0$. (Note that $[0, T]$ is
	within the time of existence for the inviscid problem---see \cref{T:InviscidExistence}). 
	Let $\nu > 0$ and assume that $\rho_0 \in L^1 \cap L^\iny$. Then there exists a weak solution
	to \GAGnu as in \cref{D:WeakSolution} on the time interval $[0, T]$ with
	\begin{align}\label{e:rhoBound1}
		\begin{split}
		\norm{\rho^\nu(t)}_{L^\iny}
			&\le \frac{\norm{\rho_0}_{L^\iny}}{1 - \abs{\sigma_2} \norm{\rho_0}_{L^\iny} t}.
		\end{split}
	\end{align}
	When $\sigma_2 \ne 0$, we have
	\begin{align}\label{e:rhoBound2}
		\begin{split}
		\norm{\rho^\nu(t)}_{L^q}
			&\le \norm{\rho_0}_{L^q}
				\pr{1 - \abs{\sigma_2} \norm{\rho_0}_{L^\iny} t}
					^{-\abs{q^{-1} \sigma_1 \slash \sigma_2 + 1}}
				\, \forall \, q \in [1, \iny), \\
		\norm{\rho^\nu(t)}^2 + 2 \nu \int_0^t \norm{\grad \rho^\nu}^2
			&\le \norm{\rho_0}^2
				\pr{1 - \abs{\sigma_2} \norm{\rho_0}_{L^\iny} t}
					^{-\abs{\sigma_1 \slash \sigma_2 + 2}}.
		\end{split}
	\end{align}
	When $\sigma_2 = 0$, we have
	\begin{align}\label{e:rhoBoundsigma2zero}
		\begin{split}
		\norm{\rho^\nu(t)}_{L^q}
			&\le \norm{\rho_0}_{L^q} \exp \pr{\abs{\sigma_1} q^{-1} \norm{\rho_0}_{L^\iny} t}
				\, \forall \, q \in [1, \iny), \\
		\norm{\rho^\nu(t)}^2 + 2 \nu \int_0^t \norm{\grad \rho^\nu}^2
			&\le \norm{\rho_0}^2 \exp \pr{\abs{\sigma_1} \norm{\rho_0}_{L^\iny} t}.
		\end{split}
	\end{align}
\end{theorem}

\begin{remark}

\noindent Uniqueness of solutions as in \cref{D:WeakSolution} is addressed in \cite{BedrossianAzzam2015} (at least for $\sigma_1 +\sigma_2 =0$). In 2D, uniqueness can also be obtained using arguments very close to those we give in \cref{S:VV,S:VVNonL2} for the vanishing viscosity limit.

\end{remark}

To prove \cref{T:ViscousExistence}, we will construct a sequence of approximations.
We will obtain the necessary bounds on this sequence in \cref{L:rhonLqBounds}, then use these bounds in the proof proper of \cref{T:ViscousExistence}.

The sequence of approximations is defined as follows:
\begin{align}\label{e:ApproxViscous}
	\begin{array}{l}
		\rho_0(t, x) = \rho_0(x), \\
		\v_n = \sigma_1 \grad \F * \rho_{n - 1}, \\
		\prt_t \rho_n + \v_n \cdot \grad \rho_n
			= \sigma_2 \rho_{n - 1} \rho_n + \nu \Delta \rho_n, \\
		\rho_n(0) = \rho_0
	\end{array}
\end{align}
for $n = 1, 2, \dots$. Note that
\begin{align*}
	\dv \v_n = \sigma_1 \rho_{n - 1}.
\end{align*}

\begin{lemma}\label{L:rhonLqBounds}
	Fix $T > 0$ with $T < (\abs{\sigma_2} \norm{\rho_0}_{L^\iny})^{-1}$ or
	$T < \iny$ if $\sigma_2 = 0$. Let $\nu \ge 0$, $n \ge 0$, $t \in [0, T]$.
	We have
	\begin{align}\label{e:rhonBound1}
		\begin{split}
		\norm{\rho_n(t)}_{L^\iny}
			&\le \frac{\norm{\rho_0}_{L^\iny}}{1 - \abs{\sigma_2} \norm{\rho_0}_{L^\iny} t}, \\
		(\prt_t \rho_n) &\text{ is bounded in } L^2(0, T; H^{-1}(\R^d)).
		\end{split}
	\end{align}
	When $\sigma_2 \ne 0$, we have
	\begin{align}\label{e:rhonBound2}
		\begin{split}
		\norm{\rho_n(t)}_{L^q}
			&\le \norm{\rho_0}_{L^q}
				\pr{1 - \abs{\sigma_2} \norm{\rho_0}_{L^\iny} t}
					^{-\abs{q^{-1} \sigma_1/\sigma_2 + 1}}
					\, \forall \, q \in [1, \iny), \\
		\norm{\rho_n(t)}^2 + 2 \nu \int_0^t \norm{\grad \rho_n}^2
			&= \norm{\rho_0}^2
				\pr{1 - \abs{\sigma_2} \norm{\rho_0}_{L^\iny} t}
					^{-\abs{\sigma_1/\sigma_2 + 2}}.
		\end{split}
	\end{align}
	When $\sigma_2 = 0$, we have
	\begin{align}\label{e:rhonBoundsigma2zero}
		\begin{split}
		\norm{\rho_n(t)}_{L^q}
			&\le \norm{\rho_0}_{L^q} \exp \pr{\abs{\sigma_1} q^{-1} \norm{\rho_0}_{L^\iny} t}
				\, \forall \, q \in [1, \iny), \\
		\norm{\rho_n(t)}^2 + 2 \nu \int_0^t \norm{\grad \rho_n}^2
			&= \norm{\rho_0}^2 \exp \pr{\abs{\sigma_1} \norm{\rho_0}_{L^\iny} t}.
		\end{split}
	\end{align}
\end{lemma}
\begin{proof}
We will start with the assumption that $\rho_0 \in L^1 \cap L^\iny \cap C^\iny$.

Set as an induction hypothesis that $\rho_n \in C(0, T; L^1 \cap L^\iny \cap C^\iny)$ and that $\cref{e:rhonBound1}_1$ holds.

This is certainly true for $n = 0$. Assume it is true up to $n - 1$.
Then by \cref{P:LinearProblem}, the equation defining $\rho_n$ in $\cref{e:ApproxViscous}_3$ has a solution in $C(0, T; L^1 \cap L^\iny \cap C^\iny)$.

Assume that $q$ is a rational number in $[2, \iny)$ with $q = m/n$ in lowest terms for $m$ even. This insures that $\rho_n^q \ge 0$. The conclusions we reach for such rational $q$'s will hold for all $q \in [2, \iny)$ by the continuity of Lebesgue norms.

Multiplying $\cref{e:ApproxViscous}_3$ by $\rho_n^{q - 1}$ and integrating gives
\begin{align*}
	\innp{\prt_t \rho_n, \rho_n^{q - 1}} + \innp{\v_n \cdot \grad \rho_n, \rho_n^{q - 1}}
		= \sigma_2 \innp{\rho_{n - 1} \rho_n, \rho_n^{q - 1}}
				+ \nu \innp{\Delta \rho_n, \rho_n^{q - 1}}.
\end{align*}
But,
\begin{align*}
	\innp{\prt_t \rho_n, \rho_n^{q - 1}}
		&= \frac{1}{q} \int_{\R^d} \prt_t \rho_n^q
		= \frac{1}{q} \diff{}{t} \norm{\rho_n}_{L^q}^q, \\
	\innp{\v_n \cdot \grad \rho_n, \rho_n^{q - 1}}
		&= \frac{1}{q} \int_{\R^d} \v_n \cdot \grad \rho_n^q
		= - \frac{1}{q} \int_{\R^d} \dv \v_n \, \rho_n^q
		= - \frac{\sigma_1}{q} \int_{\R^d} \rho_{n - 1} \rho_n^q, \\
	\sigma_2 \innp{\rho_{n - 1} \rho_n, \rho_n^{q - 1}}
		&= \sigma_2 \int_{\R^d} \rho_{n - 1} \rho_n^q, \\
	\nu \innp{\Delta \rho_n, \rho_n^{q - 1}}
		&= -\nu \innp{\grad \rho_n, \grad \rho_n^{q - 1}}
		= - (q - 1) \nu \int_{\R^d} \rho_n^{q - 2} \abs{\grad \rho_n}^2
		\le 0.
\end{align*}
Thus,
\begin{align}\label{e:PreGronwallsrhonBound}
	\diff{}{t} \norm{\rho_n}_{L^q}^q
		\le (\sigma_1 + q \sigma_2) \int_{\R^d} \rho_{n - 1} \rho_n^q
		\le
			\abs{\sigma_1 + q \sigma_2} \norm{\rho_{n - 1}}_{L^\iny} \norm{\rho_n}_{L^q}^q.
\end{align}

Now assume that $q \in [1, 2)$, with the same assumption on its rationality as before, and observe that the above argument fails since $\rho_n^{q - 2}$ is singular at $\rho_n = 0$. We therefore modify the argument as follows. Fix $\eps > 0$ and define $\la \in C^\iny(\R)$ so that
\begin{align*}
	\la(x)
		=
		\begin{cases}
			\eps, & 0 \le x \le \eps, \\
			x, & x \ge 2 \eps
		\end{cases}
\end{align*}
and so $\la', \la'' \ge 0$.
Multiplying $\cref{e:ApproxViscous}_3$ by $\la'(\rho_n^q) \rho_n^{q - 1}$ and integrating gives
\begin{align*}
	\innp{\prt_t \rho_n, \la'(\rho_n^q) \rho_n^{q - 1}}
			+ \innp{\v_n \cdot \grad \rho_n, \la'(\rho_n^q) \rho_n^{q - 1}}
		= \sigma_2 \innp{\rho_{n - 1} \rho_n, \la'(\rho_n^q) \rho_n^{q - 1}}
				+ \nu \innp{\Delta \rho_n, \la'(\rho_n^q) \rho_n^{q - 1}}.
\end{align*}
But,
\begin{align*}
	\innp{\prt_t \rho_n, \la'(\rho_n^q) \rho_n^{q - 1}}
		&= \frac{1}{q} \int_{\R^d} \prt_t \la(\rho_n^q)
		= \frac{1}{q} \diff{}{t} \int_{\R^d} \la(\rho_n^q), \\
	\innp{\v_n \cdot \grad \rho_n, \la'(\rho_n^q) \rho_n^{q - 1}}
		&= \frac{1}{q} \int_{\R^d} \v_n \cdot \grad \la(\rho_n^q)
		= - \frac{1}{q} \int_{\R^d} \dv \v_n \, \la(\rho_n^q)
		= - \frac{\sigma_1}{q} \int_{\R^d} \rho_{n - 1} \la(\rho_n^q), \\
	\sigma_2 \innp{\rho_{n - 1} \rho_n, \la'(\rho_n^q) \rho_n^{q - 1}}
		&= \sigma_2 \int_{\R^d} \rho_{n - 1} \la'(\rho_n^q) \rho_n^q, \\
	\nu \innp{\Delta \rho_n, \la'(\rho_n^q) \rho_n^{q - 1}}
		&= -\nu \innp{\grad \rho_n, \grad (\la'(\rho_n^q) \rho_n^{q - 1})} \\
		&= - (q - 1) \nu \int_{\R^d} \la'(\rho_n^q) \rho_n^{q - 2} \abs{\grad \rho_n}^2
			- q \nu \int_{\R^d} \la''(\rho_n^q) \rho_n^{2(q - 1)} \abs{\grad \rho_n}^2
		\le 0.
\end{align*}
Taking the limit as $\eps \to 0$, we recover the same bound as in \cref{e:PreGronwallsrhonBound}. We use here that $\la' = 0$ in a neighborhood of the origin so that the singularity in $\rho_n^{q - 2}$ is removed.

Let $t \in [0, T]$. Applying Gronwall's lemma gives
\begin{align*}
	\norm{\rho_n(t)}_{L^q}^q
		\le \norm{\rho_0}_{L^q}^q \exp \pr{\abs{\sigma_1 + q \sigma_2}
			\int_0^t \norm{\rho_{n - 1}(s)}_{L^\iny} \, ds}
\end{align*}
so that
\begin{align}\label{e:rhonLq}
	\norm{\rho_n(t)}_{L^q}
		\le \norm{\rho_0}_{L^q} \exp \pr{\abs{q^{-1}\sigma_1 + \sigma_2}
			\int_0^t \norm{\rho_{n - 1}(s)}_{L^\iny} \, ds}.
\end{align}

Now, by the induction hypothesis,
\begin{align*}
	\int_0^t \norm{\rho_{n - 1}(s)}_{L^\iny} \, ds
		&\le \int_0^t \frac{\norm{\rho_0}_{L^\iny}}{1 - \abs{\sigma_2}
					\norm{\rho_0}_{L^\iny} s} \, ds \\
		&= \left\{
			\begin{array}{rl}
				-\abs{\sigma_2}^{-1} \log \pr{1 - \abs{\sigma_2} \norm{\rho_0}_{L^\iny} t},
				&\sigma_2 \ne 0, \\
				\norm{\rho_0}_{L^\iny} t,
				&\sigma_2 = 0.
			\end{array}
		\right.
\end{align*}
Taking the limit as $q \to \iny$ of both sides of \cref{e:rhonLq}, it follows by the continuity of Lebesgue norms that for $\sigma_2 \ne 0$,
\begin{align*}
	\norm{\rho_n(t)}_{L^\iny}
		&\le \norm{\rho_0}_{L^\iny}
			\exp \pr{-\log \pr{1 - \abs{\sigma_2} \norm{\rho_0}_{L^\iny} t}}
		= \norm{\rho_0}_{L^\iny}
			\pr{1 - \abs{\sigma_2} \norm{\rho_0}_{L^\iny} t}^{-1}
\end{align*}
and $\norm{\rho_n(t)}_{L^\iny} \le \norm{\rho_0}_{L^\iny}$ if $\sigma_2 = 0$.
This shows that the induction hypothesis, $\cref{e:rhonBound1}_1$, holds for $n$, and so for all $n$ by induction.

Returning to \cref{e:rhonLq}, we see that $\cref{e:rhonBound2}_1$ and $\cref{e:rhonBoundsigma2zero}_1$ hold.

The bounds in $\cref{e:rhonBound2}_2$, $\cref{e:rhonBoundsigma2zero}_2$ follow by not discarding for $q = 2$ the term above that we observed was never positive. Using $\prt_t \rho_n = - \v_n \cdot \grad \rho_n + \nu \Delta \rho_n$, $\cref{e:rhonBound1}_2$ then follows from $\cref{e:rhonBound2}_2$ or $\cref{e:rhonBoundsigma2zero}_2$ along with \cref{L:vBounds}.

\Ignore{ 
To prove $\cref{e:rhonBound1}_2$, we first note that $\rho_n(t) \in L^\iny(0, T; L^1(\R^d))$. This follows from $\cref{e:rhonBound1}_1$ and the rapid spatial decay of $\rho_n$: the formal argument for this decay :::::

we integrate $\cref{e:ApproxViscous}_3$ in space, giving
\begin{align*}
	\int_{\R^d} \prt_t \rho_n + \int_{\R^d} \v_n \cdot \grad \rho_n
		= \sigma_2 \int_{\R^d} \rho_{n - 1} \rho_n + \nu \int_{\R^d}  \Delta \rho_n,
\end{align*}
so that
\begin{align*}
	\diff{}{t} \int_{\R^d} \rho_n
		= \sigma_2 \int_{\R^d} \rho_{n - 1} \rho_n +  \int_{\R^d} \dv \v_n \rho_n + \nu \int_{\R^d} \dv \grad \rho_n.
\end{align*}
The final integral vanishes and $\dv \v_n = \sigma_1 \rho_{n - 1}$, so
\begin{align*}
	\diff{}{t} m(\rho_n)
		&= (\sigma_1 + \sigma_2) \int_{\R^d} \rho_{n - 1} \rho_n
		= (\sigma_1 + \sigma_2) \innp{\rho_{n - 1}, \rho_n}.
\end{align*}
Integrating in time yields $\cref{e:rhonBound1}_2$.
} 

Because the bounds obtained depend only upon the $L^q$ norms of $\rho_0$, we see by the density of $L^1 \cap C^\iny$ in $L^1 \cap L^\iny$ that the result holds for $\rho_0$ in $L^1 \cap L^\iny$.
\end{proof}

\begin{proof}[\textbf{Proof of \cref{T:ViscousExistence}}]
	Because of the bounds in \cref{L:rhonLqBounds}, we can make the same argument for
	existence of solutions to \GAGnu as is made for the Navier-Stokes equations on
	pages 72-73 of \cite{CF1988}. That is, except that $H^1(\R^d)$ is not compactly
	embedded in $L^2(\R^d)$ because $\R^d$ is an unbounded domain.
	(The embedding of $L^2(\R^2)$ into $H^{-1}(\R^d)$ is, however, continuous,
	and no compactness is needed for this embedding.)
	We handle this lack of compactness, however, as Temam does in Remark III.3.2 of
	\cite{T2001}.
	
	We note that because $\prt_t \rho^\nu \in L^2_{loc}([0, \iny); H^{-1})$ and
	not just in $L^1_{loc}([0, \iny); H^{-1})$, $\rho^\nu$ is equal (a.e.) to a function
	continuous in $L^2$. This is what happens for the Navier-Stokes equations in 2D versus
	higher dimension, and is treated in the same manner. (See, for instance, the argument
	following (3.60) Chapter III of \cite{T2001}.)
	
	The estimates stated in \cref{T:ViscousExistence} then follow
	from taking the limit as $n \to \iny$ of the bounds obtained in
	\cref{L:rhonLqBounds}.
\Ignore{ 
we integrate $\cref{e:ApproxViscous}_3$ in space, giving
\begin{align*}
	\int_{\R^d} \prt_t \rho_n + \int_{\R^d} \v_n \cdot \grad \rho_n
		= \sigma_2 \int_{\R^d} \rho_{n - 1} \rho_n + \nu \int_{\R^d}  \Delta \rho_n,
\end{align*}
so that
\begin{align*}
	\diff{}{t} \int_{\R^d} \rho_n
		= \sigma_2 \int_{\R^d} \rho_{n - 1} \rho_n +  \int_{\R^d} \dv \v_n \rho_n + \nu \int_{\R^d} \dv \grad \rho_n.
\end{align*}
The final integral vanishes and $\dv \v_n = \sigma_1 \rho_{n - 1}$, so
\begin{align*}
	\diff{}{t} m(\rho_n)
		&= (\sigma_1 + \sigma_2) \int_{\R^d} \rho_{n - 1} \rho_n
		= (\sigma_1 + \sigma_2) \innp{\rho_{n - 1}, \rho_n}.
\end{align*}
Integrating in time yields $\cref{e:rhonBound1}_2$.
} 
\end{proof}

We used \cref{L:vBounds} above in the proof of \cref{L:rhonLqBounds} and will use it again in later sections.

\begin{lemma}\label{L:vBounds}
	Let $\w = \grad \F * f$. Then
	\begin{align*}
		\norm{\w}_{LL}
			&\le C \norm{f}_{L^1 \cap L^\iny}.
	\end{align*}
\end{lemma}
\begin{proof}
In 2D, this is Lemma 8.1 of \cite{MB2002}. It can be proved in all dimensions in a manner very similar to that of Theorem 3.1 of \cite{VishikBesov}, so we suppress the proof.
\Ignore{ 
We follow the proof of Theorem 3.1 of \cite{VishikBesov}.  First observe that, by Lemma \ref{L:MassZero}, $\| \w \|_{L^{\infty}} \leq C\norm{f}_{L^1 \cap L^\iny}$.  To estimate the modulus of continuity of $\w$, we use Bernstein's Lemma and boundedness of Calderon-Zygmund operators on $L^2$ to write, for $x$ and $y$ satisfying $|x-y|\leq 1$,
\begin{equation*}
\begin{split}
&| \w(x) - \w(y) | \leq \sum_{j=-1}^N\| \Delta_{j} \w(x) - \Delta_{-1}\w(y) \|_{L^{\infty}} + \sum_{j= N+1}^{\infty}\| \Delta_{j} \w(x) - \Delta_{j}\w(y) \|_{L^{\infty}}\\
&\qquad \leq C\sum_{j=-1}^N\| \Delta_{j} \nabla\w \|_{L^{\infty}}|x-y| + 2\sum_{j= N+1}^{\infty} \| \Delta_{j}\w \|_{L^{\infty}}\\
&\qquad \leq C\| \Delta_{-1} \nabla\w \|_{L^{2}}|x-y| + C\sum_{j=0}^N\| \Delta_{j} \nabla\w \|_{L^{\infty}}|x-y| + C \sum_{j= N+1}^{\infty} 2^{-j}\| \Delta_{j} \nabla\w \|_{L^{\infty}}\\
&\qquad \leq C\| \Delta_{-1} f \|_{L^{2}}|x-y| + C\sum_{j=0}^N\| \Delta_{j} f \|_{L^{\infty}}|x-y| + C \sum_{j= N+1}^{\infty} 2^{-j}\| \Delta_{j} f \|_{L^{\infty}}\\
&\qquad \leq C(\|  f \|_{L^{2}} + N\|  f \|_{L^{\infty}})|x-y| +  C 2^{-N}\| f \|_{L^{\infty}}.\\
\end{split}
\end{equation*}
Set $N = 1-\log_2|x-y|$.  Then for $x$ and $y$ satisfying $|x-y|\leq 1$,
\begin{equation*}
| \w(x) - \w(y) | \leq C\norm{f}_{L^1 \cap L^\iny} |x-y|(1-\log_2|x-y|).
\end{equation*}
This completes the proof.
} 
\end{proof}

\Ignore{ 
We will find the following two easy-to-establish lemmas useful in the following sections:

\ToDo{As far as I can tell, we never use the two lemmas below.  Is this correct?}
\begin{lemma}\label{L:gradvlBound}
	For all $k \ge 1$, $\nu \ge 0$, and $t \ge 0$,  if $\rho^\nu(t) \in H^{k - 1}(\R^d)$ then
	\begin{align*}
	\smallnorm{\grad^k \v^\nu(t)}
		&= \abs{\sigma_1} \smallnorm{\grad^{k - 1} \rho^\nu(t)}.	
	\end{align*}
\end{lemma}
\begin{proof}
Observe that for $k \ge 1$, if $\rho^\nu(t) \in C_C^\iny(\R^d)$ then by \cref{L:H2Lemma},
\begin{align*}
	\smallnorm{\grad^k \v^\nu}
		&= \abs{\sigma_1} \smallnorm{\grad^k \grad \F * \rho^\nu}
		= \abs{\sigma_1} \smallnorm{\grad^{k - 1}\Delta \F * \rho^\nu}
		= \abs{\sigma_1} \smallnorm{\grad^{k - 1} \rho^\nu}.
\end{align*}
The result then follows from the density of $C_C^\iny(\R^d)$ in $H^{k - 1}(\R^d)$.
\end{proof}

\begin{lemma}\label{L:H2Lemma}
	Let $f, g \in C^\iny(\R^d)$ with $\Delta f$, $\Delta g \in L^2(\R^d)$.
	Then
	$
		\innp{\grad \grad f, \grad \grad g}
			= \innp{\Delta f, \Delta g}
	$.
\end{lemma}
\begin{proof}
	This follows from integrating by parts.
\end{proof}
} 

%
%
\section{Spatial decay of viscous solutions}\label{S:SpatialDecay}

\noindent Let $a_R$ be as in \cref{D:Radial} and let $b_R(r) = 1 - a_R(x)$, where $r := \abs{x}$.

\begin{prop}\label{P:rhonuTailBound}
	Let $\nu \ge 0$. For any $R > 0$, for solutions to ($GAG_{\nu}$) satisfying Definition (\ref{D:WeakSolution}),
	\begin{align}\label{e:rhonuTailPrelimBound}
		\norm{b_R \rho^\nu(t)}^2
				+ \nu \int_0^t \norm{b_R \grad \rho^\nu}^2
			&\le \pr{\norm{b_R \rho^\nu_0}^2 + C_0(t)t(\nu + 1) R^{-2}} e^{C_0(t) t}.
	\end{align}
	Now assume that for some integer $N \ge 2$ we have
	$\norm{b_R \rho^\nu_0}^2 \le C R^{-N}$ for all sufficiently large $R$. Then
	for all sufficiently large $R$, we have
	\begin{align}\label{e:rhonuTailBound}
		\norm{b_R \rho^\nu(t)}^2
				+ \nu \int_0^t \norm{b_R \grad \rho^\nu}^2
			&\le C_0(t) t (\nu + 1)^2 R^{-N} e^{C_0(t) t}.
	\end{align}
\end{prop}
\begin{proof}
	We drop the superscript $\nu$ for convenience.
	
	Multiplying \GAGnu by $b_R^2 \rho$ and integrating, we have
	\begin{align*}
		\innp{\prt_t \rho, b_R^2 \rho}
			= - \innp{\v \cdot \grad \rho, b_R^2 \rho}
				+ \sigma_2 \innp{\rho^2, b_R^2 \rho}
				+ \nu \innp{\Delta \rho, b_R^2 \rho}.
	\end{align*}
	Now,
	\begin{align*}
		\innp{\prt_t \rho, b_R^2 \rho}
			&= \frac{1}{2} \diff{}{t} \norm{b_R \rho}^2, \\
		-\innp{\v \cdot \grad \rho, b_R^2 \rho}
			&= -\frac{1}{2} \innp{\v, \grad (b_R^2 \rho^2)}
				+ \frac{1}{2} \innp{\v, \rho^2 \grad b_R^2} \\
			&= \frac{1}{2} \innp{\dv \v, b_R^2 \rho^2}
				+ \innp{\v, \rho^2 b_R \grad b_R}
			= \frac{\sigma_1}{2} \innp{\rho, b_R^2 \rho^2}
				+ \innp{\v, \rho^2 b_R \grad b_R}, \\
		\nu \innp{\Delta \rho, b_R^2 \rho}
			&= - \nu \innp{\grad \rho, \grad(b_R^2 \rho)}
			= - \nu \innp{\grad \rho, b_R^2 \grad \rho
					+ \rho \grad (b_R)^2} \\
			&= - \nu \norm{b_R \grad \rho}^2
				- 2 \nu \innp{b_R \grad \rho, \rho \grad b_R}.
	\end{align*}
	
	We estimate the various terms, taking advantage of \cref{T:ViscousExistence,T:InviscidExistence}:
	\begin{align*}
		\abs{\innp{\rho, b_R^2 \rho^2}}
			&\le \norm{\rho}_{L^\iny} \norm{b_R^2 \rho^2}_{L^1}
			\le C_0(t) \norm{b_R \rho}^2, \\
		\abs{\innp{\v, \rho^2 b_R \grad b_R}}
			&\le \norm{\grad b_R}_{L^\iny} \norm{\v}_{L^\iny} \norm{\rho} \norm{b_R \rho}
			\le C_0(t) (R^{-2} + \norm{b_R \rho}^2),\\
		\abs{2 \nu \innp{b_R \grad \rho, \rho \grad b_R}}
			&\le \frac{\nu}{2} \norm{b_R \grad \rho}^2
				+ 2\nu \norm{\rho \grad b_R}^2,\\
		\nu \norm{\rho \grad b_R}^2
			&\le \nu C_0(t) R^{-2}, \\
		\sigma_2 \innp{\rho^2, b_R^2 \rho}
			&\le \abs{\sigma_2} \norm{\rho}_{L^\iny} \norm{b_R \rho}^2
				\le C_0(t) \norm{b_R \rho}^2.
	\end{align*}
	
	Combining the estimates above, we have
	\begin{align*}
		\frac{1}{2} \diff{}{t} &\norm{b_R \rho}^2
				+ \nu \norm{b_R \grad \rho}^2
			\le C_0(t)(\nu + 1) R^{-2}
				+ \frac{\nu}{2} \norm{b_R \grad \rho}^2
					+ C_0(t) \norm{b_R \rho}^2,
	\end{align*}
	or,
	\begin{align*}
		\diff{}{t} &\norm{b_R \rho}^2
				+ \nu \norm{b_R \grad \rho}^2
			\le C_0(t)(\nu + 1) R^{-2}
					+ C_0(t) \norm{b_R \rho}^2.
	\end{align*}
	Applying Gronwall's lemma (using that $C_0(t)$ increases in $t$) gives
	\cref{e:rhonuTailPrelimBound}.
	
	Now assume that for some integer $N \ge 2$ we have
	$\norm{b_R \rho_0}^2 \le C R^{-N}$ for all sufficiently large $R$.
	We refine some of our estimates using \cref{e:rhonuTailPrelimBound}. We have,
	\begin{align*}
		\abs{\innp{\v, \rho^2 b_R \grad b_R}}
			&\le \norm{\v}_{L^\iny} \norm{\grad b_R}_{L^\iny} \norm{b_R \rho}
				\norm{\rho}_{L^2(\supp \grad b_R)} \\
			&\le C_0(t) R^{-1} \pr{\norm{b_R \rho_0}^2
				+ C_0(t)t(\nu + 1) R^{-2}}^{1\slash 2} e^{C_0(t) t}
				\norm{\rho}_{L^2(\supp \grad b_R)},
				\\
			&\le C_0(t) R^{-1} \pr{C R^{-N} + C_0(t)t(\nu + 1) R^{-2}}^{1\slash 2} e^{C_0(t) t}
				\norm{\rho}_{L^2(\supp \grad b_R)}
	\end{align*}
	and
	\begin{align*}
		\nu \norm{\rho \grad b_R}^2
			&\le \nu \norm{\grad b_R}_{L^\iny}^2 \norm{\rho}_{L^2(\supp \grad b_R)}^2
			\le \nu R^{-2} \norm{\rho}_{L^2(\supp \grad b_R)}^2.
	\end{align*}
	But
	\begin{align*}
		\norm{\rho}_{L^2(\supp \grad b_R)}
			&\le \norm{b_{R/2} \rho}
			\le \pr{\norm{b_{R/2} \rho_0}^2
				+ C_0(t)t(\nu + 1) R^{-2}}^{1\slash 2} e^{C_0(t) t} \\
			&\le \pr{C R^{-N} + C_0(t)t(\nu + 1) R^{-2}}^{1\slash 2} e^{C_0(t) t}
	\end{align*}
	follows from \cref{e:rhonuTailPrelimBound} applied with $R/2$ in place of $R$.
	
	We conclude that
	\begin{align*}
		\diff{}{t} &\norm{b_R \rho}^2
				+ \nu \norm{b_R \grad \rho}^2
			\le C_0(t)(\nu + 1)^2 (R^{-N-1} + R^{-3})
					+ C_0(t) \norm{b_R \rho}^2 
	\end{align*}
	holds for all sufficiently large $R$. Integrating in time and applying
	Gronwall's inequality shows, in particular, that \cref{e:rhonuTailBound} holds for $N \leq 3$.
	Applying the above process inductively gives \cref{e:rhonuTailBound} for all
	$N \ge 2$.
\end{proof}

\Ignore{ 
\begin{cor}\label{C:R7Decay}
	Assume that $\norm{b_R \rho_0}^2 \le C R^{-(d+4)}$
	for all sufficiently large $R$.
	Then $x \rho^\nu(x) \in L^1(\R^d)$.
\end{cor}
\begin{proof}
	Letting $A_R := B_R(0) \setminus B_{R - 1}(0)$, we have
	\begin{align*}
		\norm{x \rho^\nu}_{L^1}
			&= \sum_{R = 1}^\iny \norm{x \rho^\nu}_{L^1(A_R)}
			\le \sum_{R = 1}^\iny R \norm{\rho^\nu}_{L^1(A_R)}
			\le C\sum_{R = 1}^\iny R^{1+\frac{d-1}{2}} \norm{\rho^\nu}_{L^2(A_R)} \\
			&\le C\sum_{R = 1}^\iny R^{\frac{d+1}{2}} \norm{b_{R/2} \rho^\nu}
			\le C_0(t) (\nu + 1)^{\frac{1}{2}}
				\sum_{R = 1}^\iny R^{\frac{d+1}{2}} R^{-(d+4)/2}
			< \iny.
	\end{align*}
\end{proof}
} 

\begin{cor}\label{C:L1Membership}
	Fix $\al \ge 0$ and suppose that for some integer $N > d + 2\al + 1$ and some $R_0>0$, we have
	$\norm{b_R \rho^\nu_0}^2 \le C R^{-N}$ for all $R \ge R_0$. Then
	$\abs{x}^\al \rho^\nu(t, x) \in L^1_x(\R^d)$ up to the time of existence.
\end{cor}
\begin{proof}
	For convenience, we set $\rho = \rho^\nu$. Then
	\begin{align*}
		&\norm{\abs{x}^\al \rho}_{L_x^1(B_{R_0}^C)}
			\le \sum_{k = 1}^\iny \norm{\abs{x}^\al \rho}_{L^1(B_{(k + 1) R_0} \setminus B_{k R_0})}
			\le CR_0^{\al + \frac{d}{2}} \sum_{k = 1}^\iny (k + 1)^{\alpha + \frac{d-1}{2}}  
					\norm{\rho}_{L^2(B_{(k + 1) R_0} \setminus B_{k R_0})} \\
			&\qquad
			\le C R_0^{\al + \frac{d}{2}} \sum_{k = 1}^\iny k^{\al + \frac{d-1}{2}}
					\norm{b_{\frac{k R_0}{2}} \rho}
			\le C(t, \nu, \al, R_0) \sum_{k = 1}^\iny k^{\al + \frac{d-1}{2}}
					\left(\frac{R_0 k}{2}\right)^{-\frac{N}{2}} \\
			&\qquad
			\le C(t, \nu, \al, R_0) \sum_{k = 1}^\iny k^{\frac{2 \al + d -1 - N}{2}} < \infty.
	\end{align*}
\end{proof}

\begin{theorem}
	Assume that for all $R \ge R_0$,
	$\norm{b_R \rho^\nu_0}^2 \le C R^{-N}$ for some $N > d + 1$.
	Then up to the time of existence,
	\begin{align}\label{e:TotalMassId}
		m(\rho^\nu)
			&= m(\rho_0) + (\sigma_1 + \sigma_2) \int_0^t \norm{\rho^\nu(s)}^2 \, ds.
	\end{align}
\end{theorem}
\begin{proof}
	By \cref{C:L1Membership}, $\rho^\nu(t, x) \in L^1_x(\R^d)$ up to the time of existence.
	We integrate \GAGnu over $\R^d$ to obtain
	\begin{align*}
		\int_{\R^d} \prt_t \rho^\nu
			+ \int_{\R^d} \v^\nu \cdot \grad \rho^\nu
			= \sigma_2 \int_{\R^d} (\rho^\nu)^2
				+ \nu \int_{\R^d} \Delta \rho^\nu.
	\end{align*}
	
	The bound in
	\cref{e:rhonuTailPrelimBound} is sufficient to integrate one term by parts, giving
	\begin{align*}
		\int_{\R^d} \v^\nu \cdot \grad \rho^\nu
			= - \int_{\R^d} \dv \v^\nu \rho^\nu
			= - \sigma_1 \int_{\R^d} (\rho^\nu)^2.
	\end{align*}
	For the other two terms we have, formally,	
	\begin{align}\label{e:FormalMassIBPs}
		\int_{\R^d} \prt_t \rho^\nu
			= \diff{}{t} \int_{\R^d} \rho^\nu, \quad
		\int_{\R^d} \Delta \rho^\nu
			= 0.
	\end{align}
	Integrating in time then yields \cref{e:TotalMassId}.
	
	Our weak solutions, however, lack the time regularity to obtain the first equality
	in \cref{e:FormalMassIBPs},
	and the spatial regularity and decay to obtain the second. 
	To justify these equalities, we could mollify the initial data and employ a sequence of
	approximate solutions. Alternately, we could obtain the equivalent of
	\cref{P:rhonuTailBound,C:L1Membership} for the sequence, $(\rho_n)$, of approximate
	solutions employed
	in the proof of \cref{T:ViscousExistence}. This would lead to the identity,
	\begin{align*}
		m(\rho_n)
			&= m(\rho_0) + (\sigma_1 + \sigma_2) \int_0^t \innp{\rho_{n - 1}(s), \rho_n(s)} \, ds,
	\end{align*}
	which in turn yields \cref{e:TotalMassId} in the limit as $n \to \iny$.
\end{proof}

When $\sigma_1 = - \sigma_2$, as happens for \PAGnu, total mass is conserved, as we can see from \cref{e:TotalMassId}.

%
%
\section{Total mass and infinite energy}\label{S:InfiniteEnergy}

\noindent In dimensions three and higher, $\rho^\nu \in L^1 \cap L^\iny$ is enough to guarantee membership of $\v^\nu$ to $L^2(\R^d)$. In 2D, however, this is no longer true: the viscous (and inviscid) velocity in 2D will generically have infinite energy, even if it has finite energy at time zero (see, for example, Proposition 3.1.1 of \cite{C1998}).
When dealing only with existence of solutions to \GAGnu, the infinite energy of 2D velocities is a minor issue. We will need to face this issue directly, however, in \cref{S:VV} when we take the vanishing viscosity limit.

In recovering the velocity from its divergence, the total mass (see \cref{e:MassDef}) of the density plays an important, if so far hidden, role in 2D: in short, if the total mass of the density is zero and has sufficient spatial decay, then the velocity will lie in $L^2$.  We prove this, along with other useful bounds on the velocity, in \cref{L:MassZero}.

Before giving the proof of \cref{L:MassZero}, we must first define the Littlewood-Paley operators.  It is classical that there exists two functions ${\chi}, {\phi} \in S(\R^d)$ with supp $\hat{\chi}\subset \{\xi\in \R^d: |\xi |\leq \frac{5}{6} \}$ and supp $\hat{\phi}\subset \{\xi\in \R^d: \frac{3}{5} \leq|\xi |\leq \frac{5}{3} \}$, such that, if for every $j\geq 0$ we set $\phi_j(x) = 2^{jd} \phi(2^j x)$, then
\begin{equation*}
\begin{split}
	&\hat{\chi}+ \sum_{j\geq 0} \hat{\phi_j}
		= \hat{\chi} + \sum_{j\geq 0} \hat{\phi}(2^{-j} \cdot) 
		\equiv 1.
\end{split}
\end{equation*}

\Ignore{For $n\in\Z$ define ${\chi}_n \in S(\R^d)$ in terms of its Fourier transform ${\hat{\chi}}_n$, where ${\hat{\chi}}_n$ satisfies 
\begin{equation*}
{\hat{\chi}}_n (\xi) =   \hat{\chi}(\xi) + \sum_{j\leq n} \hat{\phi_j}(\xi)
\end{equation*}
for all $\xi\in\R^d$. }

For $f\in S'(\R^d)$ and $j\geq -1$, define the Littlewood-Paley operators ${\Delta}_j$ by 
\begin{align*}
	\Delta_j f  = \quad
	\left\{
	\begin{array}{l}
		\chi\ast f,  \qquad j=-1 \\
		 \phi_j\ast f, \qquad j\geq 0.
	\end{array}
	\right.
\end{align*}

We make use of the following lemma throughout the paper.  A proof of the lemma can be found in \cite{C1996}, chapter 2.  Below, $C_{a,b}(0)$ denotes the annulus with inner radius $a$ and outer radius $b$.
\begin{lemma}\label{bernstein}
(Bernstein's Lemma) Let $r_1$ and $r_2$ satisfy $0<r_1<r_2<\infty$, and let $p$ and $q$ satisfy $1\leq p \leq q \leq \infty$. There exists a positive constant $C$ such that for every integer $k$ , if $u$ belongs to $L^p(\R^d)$, and supp $\hat{u}\subset B_{r_1\lambda}(0)$, then 
\begin{equation}\label{bern1}
\sup_{|\alpha|=k} ||\partial^{\alpha}u||_{L^q} \leq C^k{\lambda}^{k+d(\frac{1}{p}-\frac{1}{q})}||u||_{L^p}.
\end{equation}
Furthermore, if supp $\hat{u}\subset C_{r_1\lambda, r_2\lambda}(0)$, then 
\begin{equation}\label{bern2}
C^{-k}{\lambda}^k||u||_{L^p} \leq \sup_{|\alpha|=k}||\partial^{\alpha}u||_{L^p} \leq C^{k}{\lambda}^k||u||_{L^p}.
\end{equation} 
\end{lemma} 

The following Littlewood-Paley definition of Holder spaces will be useful.  This definition is equivalent to the classical definition of Holder spaces when $\alpha$ is a positive non-integer (see, for example, \cite{C1996}, chapter 2).
\begin{definition}\label{D:HolderSpaces}
For $\alpha\in\R$, the space $C_*^{\alpha}$ is the set of functions $f$ such that 
\begin{equation*}
\sup_{j\geq -1} 2^{j\alpha}\| \Delta_j f \|_{L^{\infty}} < \infty. 
\end{equation*}
We set 
\begin{equation*}
\| f \|_{C_*^{\alpha}} = \sup_{j\geq -1}2^{j\alpha}\| \Delta_j f \|_{L^{\infty}}.
\end{equation*}
When $\alpha$ is a positive non-integer, we will often write $C^{\alpha}$ in place of $C^{\alpha}_*$, in view of the equivalence between the two spaces.
\end{definition}

\begin{lemma}\label{L:MassZero}
	Let $\rho \in L^1 \cap L^\iny(\R^d)$.
	For all $p \in (d/(d - 1), \iny]$, $d \ge 2$,
	\begin{align}\label{e:gradFrhobound}
		\norm{\grad \F * \rho}_{L^p}
			\le C \norm{\rho}_{L^1 \cap L^p}.
	\end{align}  
	If $d \ge 3$ then for all $p \in (d/(d - 2), \iny]$,
	\begin{align}\label{e:CalFmuBound}
		\norm{\F * \rho}_{L^p}
			\le C(p) \norm \rho_{L^1 \cap L^p}, \quad
		\norm{\grad \F * \rho}
			\le C \norm \rho_{L^1 \cap L^\iny}^{\frac{1}{2}} \norm{\rho}_{L^1}^{\frac{1}{2}}.
	\end{align}  
	Moreover, for $d \ge 3$, let $p_1, p_2 \in \R$ with $1 \le p_1 < d/2 < p_2 \le \iny$.  Then
	\begin{align}\label{e:CalFmuBoundLp}
		\norm{\F * \rho}_{L^\iny}
			\le C \norm \rho_{L^{p_1} \cap L^{p_2}}, \quad
		\norm{\grad \F * \rho}
			\le C \norm \rho_{L^{p_1} \cap L^{p_2}}^{\frac{1}{2}}
				\norm{\rho}_{L^1}^{\frac{1}{2}}.
	\end{align}
	Let $d = 2$ with $m(\rho) = 0$ and $|x|^{\epsilon} \rho(x) \in L_x^1(\R^2)$ for some $\epsilon\in(0,1]$.
	For all $p \in (\frac{2}{\epsilon},\infty]$,
	$\F * \rho \in L^p$ and
	\begin{align}\label{e:CalFmuBoundd2}
		\norm{\F * \rho}_{L^p}
			\le C(\| |x|^{\epsilon}\rho(x)\|_{L_x^1}+ \| \rho \|_{L^1 \cap L^{\infty}}), \quad
		\norm{\grad \F * \rho}^2
			\le C(\| |x|^{\epsilon}\rho(x)\|_{L^1_x}+ \| \rho \|_{L^1\cap L^{\infty}})\norm{\rho}_{L^1}.
	\end{align} 
\end{lemma}
\begin{proof}
	Let $a$ be as in Definition \ref{D:Radial}.
	First observe that
	\begin{align*}
		\norm{\grad \F * \rho}_{L^p}
			&\le \norm{a \grad \F}_{L^1} \norm{\rho}_{L^p}
					+ \norm{(1 - a) \grad \F}_{L^p} \norm{\rho}_{L^1}
			< \iny
	\end{align*}
	for all $p > d/(d - 1)$, giving \cref{e:gradFrhobound}.
	For $d \ge 3$, 
	\begin{align*}
		\norm{\F * \rho}_{L^p}
			&\le \norm{a \F}_{L^1} \norm{\rho}_{L^p}
					+ \norm{(1 - a) \F}_{L^p} \norm{\rho}_{L^1}
			< \iny
	\end{align*}
	for all $p > d/(d - 2)$. In particular,
	$\norm{\F * \rho}_{L^\iny} \le C \norm{\rho}_{L^1 \cap L^\iny}$.
	Hence, for $d \ge 3$, we can apply \cref{L:NearL2Norm}, which gives
	\begin{align*}
		\begin{split}
		\norm{\grad \F * \rho}^2
			&= - \int_{\R^d} (\F * \rho) (\Delta \F * \rho)
			\le \norm{\F * \rho}_{L^\iny} \norm{\Delta \F * \rho}_{L^1}
			= \norm{\F * \rho}_{L^\iny} \norm{\rho}_{L^1},
		\end{split}
	\end{align*}
	which leads to \cref{e:CalFmuBound}. 
	
	For the remainder of the proof, let $\FTF$ denote the Fourier transform operator.
	
To establish the lemma for $d=2$, we will first show that since $|x|^{\epsilon}\rho(x)$ belongs to $L_x^1(\R^2)$, $\hat{\rho}$ belongs to $C^{\epsilon}(\R^2)$.  To see this, note that for $j\geq 0$,
\begin{equation*}
\begin{split}
&\| \Delta_j \hat{\rho} \|_{L^{\infty}} = \| \phi_j * \hat{\rho} \|_{L^{\infty}} \leq  \| \FTF(\phi_j * \hat{\rho}) \|_{L^{1}} = \| \hat{\phi}_j (x)\rho(-x) \|_{L_x^{1}}= \| \hat{\phi}_j(x) |x|^{-\epsilon}|x|^{\epsilon}\rho(-x) \|_{L_x^{1}}\\
&\qquad  \leq \| \hat{\phi}_j(x) |x|^{-\epsilon} \|_{L_x^{\infty}(\text{supp }\hat\phi_j  )} \| |x|^{\epsilon}\rho(x) \|_{L_x^{1}} \leq C2^{-j\epsilon} \| |x|^{\epsilon}\rho(x) \|_{L_x^{1}}.
\end{split}
\end{equation*}	
Also note that, by Bernstein's Lemma (or Young's convolution inequality), $\| \Delta_{-1} \hat{\rho}\|_{L^{\infty}}\leq C\| \hat{\rho} \|_{L^2} = C\| \rho \|_{L^2}$.  We conclude that 
\begin{equation*}
\begin{split}
&\| \hat{\rho} \|_{C^{\epsilon}} = \sup_{j\geq -1} 2^{j\epsilon} \| \Delta_j \hat{\rho}\|_{L^{\infty}} \leq C\| \rho \|_{L^2} + \sup_{j\geq 0} 2^{j\epsilon} \| \Delta_j \hat{\rho}\|_{L^{\infty}} \\
&\qquad\qquad \leq C(\| \rho \|_{L^2} + \| |x|^{\epsilon}\rho(x) \|_{L_x^{1}}).
\end{split}
\end{equation*} 	
Since $\hat{\rho}(0) = 0$, we can write
\begin{equation}\label{rhohatnear0}
\begin{split}
&|\hat{\rho}(\xi)| = |\hat{\rho}(\xi) - \hat{\rho}(0)| \leq C(\| \rho \|_{L^2} + \| |x|^{\epsilon}\rho \|_{L_x^{1}}) | \xi - 0 |^{\epsilon}\\
&\qquad\qquad  = C(\| \rho \|_{L^2} + \| |x|^{\epsilon}\rho(x) \|_{L_x^{1}}) | \xi |^{\epsilon} 
\end{split}
\end{equation}  
for each $\xi\in\R^2$.  The estimate (\ref{rhohatnear0}) implies that $\FTF({\Delta_{-1}(\F\ast\rho)})$ belongs to $L^r(\R^2)$ for all $r<\frac{2}{2-\epsilon}$.  Thus, using the Hausdorff-Young inequality, for any $p$ satisfying $1\slash p + 1\slash q = 1$, with $q<\frac{2}{2-\epsilon}$, we can write
\begin{equation*}
\begin{split}
&\| \F \ast \rho \|_{L^p} \leq \| \Delta_{-1}(\F\ast\rho) \|_{L^p} + \sum_{j\geq 0} \| \Delta_j (\F\ast\rho) \|_{L^p}\\
&\qquad \leq \| \FTF({\Delta_{-1}(\F\ast\rho)}) \|_{L^q} + \sum_{j\geq 0} 2^{-2j}\| \Delta_j (\partial_k\partial_l\F\ast\rho) \|_{L^p}\\
&\qquad \leq \| \FTF({\Delta_{-1}(\F\ast\rho)} )\|_{L^q} + C\| \rho \|_{L^p} \leq C(\| |x|^{\epsilon}\rho(x)\|_{L_x^1} + \| \rho \|_{L^1 \cap L^{\infty}}),
\end{split}
\end{equation*}   
where we also used Bernstein's Lemma to get the second inequality.  To obtain the third inequality, for $p<\infty$ we used boundedness of the Riesz transforms on $L^p$, and for $p=\infty$ we used a classical lemma (for a proof, see Lemma 4.2 of \cite{CK2006}). 

We conclude, in particular, that when $d=2$, $\norm{\F * \rho}_{L^\iny} \leq C(\| |x|^{\epsilon}\rho(x)\|_{L_x^1} + \| \rho \|_{L^1\cap L^{\infty}})$.  Since $\F \ast \rho$ belongs to $L^{p_1}$ for some $p_1<\infty$, and $\nabla\F \ast \rho$ belongs to $L^{p_2}$ for all $p_2>2$ by \cref{e:gradFrhobound}, we can apply \cref{L:NearL2Norm}, which gives
	\begin{align*}
		\begin{split}
		\norm{\grad \F * \rho}^2
			&\le  \norm{\F * \rho}_{L^\iny} \norm{\rho}_{L^1},
		\end{split}
	\end{align*}
	which leads to \cref{e:CalFmuBoundd2}.  
	\Ignore{ 
	Now,
	\begin{align*}
		\frac{1}{p_1} + \frac{1}{p_2}
			= \frac{d - 2}{d} + \frac{d - 1}{d}
			= 2 - \frac{3}{d}
			> 1
	\end{align*}
	for all $d \ge 3$. Thus, we can apply
	} 
	
	For \cref{e:CalFmuBoundLp} for $d \ge 3$, instead of the
	bound $\norm{\F * \rho}_{L^\iny} \le C \norm{\rho}_{L^1 \cap L^\iny}$,
	we use the bound
	\begin{align*}
		\norm{\F * \rho}_{L^\iny}
			&\le \norm{a \F}_{L^{p_2'}} \norm{\rho}_{L^{p_2}}
					+ \norm{(1 - a) \F}_{L^{p_1'}} \norm{\rho}_{L^{p_1}},
	\end{align*} 
	where the primes represent \Holder conjugates. This is finite since
	$p_2' < d/(d - 2)$ and $\F \in L^{p_2'}_{loc}(\R^d)$, while
	$p_1' > d/(d - 2)$ and $\F$ decays like an $L^{p_1'}$-function.
\end{proof}

To treat densities in $\R^2$ having nonzero total mass, as we will need to do in \cref{S:VVNonL2}, we will subtract from the associated velocity field a radially symmetric velocity field, $\btau_0$. We do this in analogy with the definition of the stationary solution to the Euler equations used to obtain the radial-energy decomposition of a 2D velocity field in \cite{C1998,MB2002}.

\begin{definition}\label{D:btau0}
Fix a radially symmetric function $g_0 \in C_C^\iny(\R^2)$ having total mass 1. We will abuse notation by writing both $g_0(x)$ and $g_0(r)$, where $x \in \R^2$ and $r = \abs{x}$. Define
\begin{align*}
	\btau_0(x) = f(r) x, \quad
		f(r) := \frac{1}{r^2} \int_0^r \eta g_0(\eta) \, d \eta.
\end{align*}
\end{definition}
Being a radially directed vector field, $\btau_0$, is a gradient. We see that
\begin{align*}
	\dv \btau_0
		&= 2f + x^i \prt_i f
		= 2f + x^i \frac{x^i}{r} \prt_r f
		= 2f + r \prt_r f \\
		&= 2f - 2r \frac{1}{r^3} \int_0^r \eta g(\eta) \, d \eta
			+ r \frac{r g_0(r)}{r^2}
		= 2f - 2f + g_0(r)
		= g_0(r).
\end{align*}
Hence, also, $\btau_0 = \grad \F * g_0$.

We used the following technical lemma in the proof of \cref{L:MassZero}, above.

\begin{lemma}\label{L:NearL2Norm}
	Let $\varphi \in L^{p_1} \cap L^\iny(\R^d)$ with $\grad \varphi \in L^{p_2} \cap L^\iny(\R^d)$
	and $\Delta \varphi \in L^1 \cap L^\iny(\R^d)$.
	If $1/p_1 + 1/p_2 \ge (d-1)/d$ then $\grad \varphi \in L^2(\R^d)$.
	If $1/p_1 + 1/p_2 > (d-1)/d$ then
	\begin{align}\label{e:gradvarphiBound}
		\norm{\grad \varphi}^2
			&= - \int_{\R^d} \varphi \Delta \varphi.
	\end{align}
	\Ignore{ 
	Case 2: If $1/p_1 + 1/p_2 = (d-1)/d$ then
	\begin{align}\label{e:gradvarphiBoundp2}
		\norm{\grad \varphi}^2
			\le 2 \norm{\varphi}_{L^\iny} \norm{\Delta \varphi}_{L^1}
				+ C \norm{\varphi}_{L^\iny}^2.
	\end{align}
	} 
\end{lemma}
\begin{proof}
	Let $a_R$ be as in \cref{D:Radial}.
	Assume first that
	$\varphi$ is also in $C^\iny(\R^d)$. Then
	\begin{align*}
		\int_{\R^d} \abs{\grad \varphi}^2
			&= \lim_{R \to \iny} \int_{\R^d} a_R \grad \varphi \cdot \grad \varphi
			= - \lim_{R \to \iny} \int_{\R^d} \dv(a_R \grad \varphi) \, \varphi \\
			&= - \lim_{R \to \iny} \int_{\R^d} a_R \Delta \varphi \, \varphi
			- \lim_{R \to \iny} \int_{\R^d} (\grad a_R \cdot \grad \varphi) \varphi \\
			&= - \int_{\R^d} \Delta \varphi \, \varphi
			- \lim_{R \to \iny} \int_{\R^d} (\grad a_R \cdot \grad \varphi) \varphi.
	\end{align*}
	For the first equality, the properties of $a$ allow us to apply the monotone convergence
	theorem (we may obtain $\iny$, though).
	The one limit we evaluated is valid because $a_R \Delta \varphi \to \Delta \varphi$
	in $L^1(\R^d)$.
	For the remaining limit, we have
	\begin{align}\label{e:RemainingLimit}
		\orgabs{\int_{\R^d} (\grad a_R \cdot \grad \varphi) \varphi}
			&\le \norm{\grad a_R}_{L^\iny}
				\norm{1}_{L^p(\supp a_R)}
				\norm{\grad \varphi}_{L^{p_1}}
				\norm{\varphi}_{L^{p_2}}
			\le \frac{C}{R} R^{\frac{d}{p}}
			= C R^{\frac{d}{p} - 1}.
	\end{align}
	By assumption,
	$
		1 = \frac{1}{p} + \frac{1}{p_1} + \frac{1}{p_2}
			\ge \frac{1}{p} + \frac{d-1}{d}
	$,
	so $p \ge d$. If $
		 \frac{1}{p_1} + \frac{1}{p_2}
			> \frac{d-1}{d}
	$, so that $p > d$, then we conclude that the remaining limit vanishes,
	from which \cref{e:gradvarphiBound} follows.
	\Ignore{ 

	Now suppose that $p = 2$. Then the calculation above with $\limsup$ in place
	of $\lim$ shows that $\int_{\R^d} a_R \grad \varphi \cdot \grad \varphi$ is bounded
	in $R$. But it is also nondecreasing in $R$ and hence has a limit. We conclude therefore,
	again by the monotone convergence theorem, that $\grad \varphi \in L^2(\R^d)$.
	Thus, we can return to \cref{e:RemainingLimit} using $p_1 = 2$, $p_2 = \iny$
	to conclude that
	\begin{align*}
		\norm{\grad \varphi}^2
			\le \orgabs{\int_{\R^d} \varphi \Delta \varphi}
				+ C \norm{\grad \varphi} \norm{\varphi}_{L^\iny}
			\le \norm{\varphi}_{L^\iny} \norm{\Delta \varphi}_{L^1}
				+ C \norm{\grad \varphi} \norm{\varphi}_{L^\iny}.
	\end{align*}
	This inequality is of the form $A^2 \le B + D A$, where $B, D > 0$. Applying the quadratic formula
	to bound $A$ and squaring the result gives $A^2 \le D^2 + 2 B$, which is
	\cref{e:gradvarphiBoundp2}.

	This gives the result
	for smooth $\varphi$ and in then in full generality by density.
	} 
\end{proof}


\Ignore{ 
Fix a radially symmetric function $g_0 \in C_C^\iny(\R^d)$ having total mass 1. We will abuse notation by writing both $g_0(x)$ and $g_0(r)$, where $x \in \R^d$ and $r = \abs{x}$. Define
\begin{align*}
	\btau_0(x) = \frac{x}{r^d} \int_0^r \eta^{d - 1} g_0(\eta) \, d \eta.
\end{align*}
Being a radially directed vector field, $\btau_0$, is a gradient. Also, letting
\begin{align*}
	f(r)
		= \frac{1}{r^d} \int_0^r \eta^{d - 1} g_0(\eta) \, d \eta,
\end{align*}
so that $\btau_0(x) = x f(r)$, we see that
\begin{align*}
	\dv \btau_0
		&= d f + x^i \prt_i f
		= d f + x^i \frac{x^i}{r} \prt_r f
		= d f + r \prt_r f \\
		&= df - d r \frac{1}{r^{d + 1}} \int_0^r \eta^{d - 1} g(\eta) \, d \eta
			+ r \frac{r^{d - 1} g_0(r)}{r^d}
		= df - df + g_0(r)
		= g_0(r).
\end{align*}
Hence, also, $\btau_0 = \grad \F * g_0$.
} 


\section{The inviscid problem}\label{S:InviscidProblem}

\noindent Well-posedness of weak solutions to \PAGzero locally in time having bounded, compactly supported density as well as classical solutions having \Holder regularity is proved in \cite{BLL2012}. All the solutions constructed were also Lagrangian solutions. The approach in \cite{BLL2012} can be adapted to apply to the more general equations in \GAGzero for initial density in $L^1 \cap L^\iny$, and lead to \cref{T:InviscidExistence,T:InviscidStrong}, below. Alternately, the economical and elegant proof of the existence and uniqueness of 2D solutions to the Euler equations given by Marchioro and Pulvirenti in \cite{MP1994}, which originates in their earlier text \cite{MP1984}, can be adapted to obtain the same results. 

In brief, the authors of \cite{BLL2012} first construct smooth solutions then use a sequence of approximate smooth solutions to obtain a weak solution by demonstrating convergence of the flow maps (as in \cite{MB2002}). This approach is reversed in \cite{MP1994}, where weak (Lagrangian) solutions are first constructed by obtaining the convergence of a sequence of flow maps for approximating linearizations of the 2D Euler equations. A very simple argument then shows that regularity of the initial data is propagated over time. Considerable complications arise when adapting Marchioro and Pulvirenti's to apply to \GAGzero, because the underlying velocity field is not divergence-free (analogous complications are dealt with in \cite{BLL2012}). This requires the assumption of some regularity on the initial data to obtain weak solutions, an assumption that is only removed a posteriori via a separate (but very similar) iteration to that used to prove existence. Since the focus of this paper is on the vanishing viscosity limit, we do not give the details of this alternate approach here. 

Formally, if $\rho = \rho^0$ solves \GAGzero and $X$ is the flow map for $\v = \v^0$, then
\begin{align*}
	\diff{}{t} \rho(t, X(t, x))
		= \sigma_2 \rho(t, X(t, x))^2.
\end{align*}
Integrating along flow lines gives
\begin{align}\label{e:rhoAlongFlow}
	\rho(t, X(t, x))
		&= \frac{\rho_0(x)}{1 - \sigma_2 t \rho_0(x)}.
\end{align}
This motivates the following definition of a Lagrangian solution to \GAGzero:

\begin{definition}\label{D:LagrangianSolution}
	Let $\rho \in L^\iny_{loc}([0, \iny); L^1 \cap L^\iny(\R^d)) \cap C([0, \iny); L^2(\R^d))$
	and let $\v := \sigma_1\grad \F * \rho$.
	By \cref{L:vBounds},
	$\v \in C([0, \iny); LL(\R^d))$, where $LL(\R^d)$ is the space of bounded
	log-Lipschitz vector fields, and so $\v$ has a unique classical flow map, $X$.
	We say that $\rho$ is a Lagrangian solution to the inviscid aggregation equations \GAGzero
	with initial density $\rho_0 \in L^1 \cap L^\iny$ if
	\begin{align*}
		\rho(t, x) = \frac{\rho_0(X^{-1}(t, x))}{1 - \sigma_2 t \rho_0(X^{-1}(t, x))}
	\end{align*}
	for all $t \ge 0$, $x \in \R^d$.
\end{definition}

The form of $\rho$ in \cref{e:rhoAlongFlow} also yields a sharp time of existence for our Lagrangian solutions. If we do not consider the sign of $\rho_0$, we obtain an upper limit on the existence time that is the same as that for viscous solutions in \cref{T:ViscousExistence}. Hence, we should expect that if, say, $\sigma_2 < 0$ and $\rho_0 > 0$, so that the inviscid solution exists for all time, then the existence time for viscous solutions might be considerably longer than the bound given in \cref{T:ViscousExistence}. An open question is whether for all sufficiently small viscosity, viscous solutions to \GAGnu exist for as long as the inviscid solution exists, as was established for the 3D Navier-Stokes and Euler equations in \cite{Constantin1986}. (Issues of existence times of viscous solutions in relation to the total mass of $\rho_0$ have been well-studied: see \cite{Perthame2007}.)

We have the existence of weak and of strong solutions to \GAGzero:

\begin{theorem}\label{T:InviscidExistence}
	Fix $T > 0$ with $T < (\abs{\sigma_2} \norm{\rho_0}_{L^\iny})^{-1}$ or
	$T < \iny$ if $\sigma_2 = 0$.
	Assume that $\rho_0 \in L^1 \cap L^\iny(\R^d)$ 
	is compactly supported.
	Then there exists a unique weak solution
	to \GAGzero as in \cref{D:WeakSolution} on the time interval $[0, T]$.
	This weak solution is the unique Lagrangian solution. 
	Moreover, \cref{e:rhoBound1} holds.
\end{theorem}

We establish first that with additional regularity of the initial density, a weak solution is a classical solution.

\begin{theorem}\label{T:InviscidStrong}
	Assume that $\rho_0 \in C^{k, \al}(\R^d)$ and compactly supported,
	$k \ge 0$, $\al \in (0, 1)$. There exists a unique classical solution,
	$\rho \in L^\iny(0, T; C^{k, \al})$, to \GAGzero. Moreover,
	\begin{align}\label{e:rhoCaltBound}
		\norm{\rho(t)}_{C^{k,\al}}
			\le C(t, \abs{\sigma_1}, \norm{\rho_0}_{L^1}, \norm{\rho_0}_{C^{k,\al}}).
	\end{align}
\end{theorem}

\section{The vanishing viscosity limit for \GAGnu for velocities in $L^2$}\label{S:VV}

\noindent In this section we consider the vanishing viscosity limit \VV (see \cref{S:Introduction}) for any $\sigma_1$, $\sigma_2$ when $d \ge 3$ and $\sigma_1 + \sigma_2 = 0$ when $d = 2$. In both these cases, $\v^\nu - \v^0$ remains in $L^2(\R^d)$. In \cref{S:VVNonL2} we consider the general situation in 2D.

For the remainder of this paper, we will assume that the initial density is compactly supported. This gives, through the results of \cref{S:SpatialDecay}, rapid spatial decay of the viscous solutions and no difficulties obtaining the identities for the total mass of the density (recall the definition of the total mass in \cref{e:MassDef}.) In particular, we have \cref{P:MassDiffBound}.

\begin{prop}\label{P:MassDiffBound}
	Assume that $\rho_0 \in L^1 \cap L^\iny(\R^d)$ is compactly supported.
	Let $\nu \ge 0$ and let $\mu = \rho^\nu - \rho^0$. Then
	\begin{align*}
		m(\mu(t)) &= (\sigma_1 + \sigma_2) \int_0^t \innp{\mu(s),  \rho^0(s) + \rho^\nu(s)} \, ds, \\
		\abs{m(\mu(t))}
			&\le \abs{\sigma_1 + \sigma_2} \int_0^t \pr{\norm{\rho^0(s)} + \norm{\rho^\nu(s)}}
					\norm{\mu(s)} \, ds.
	\end{align*}
\end{prop}
\begin{proof}
	This follows from \cref{e:TotalMassId}.
\end{proof}

Though total mass of $\rho^\nu - \rho^0$ is zero at time zero, there is no reason to expect, based upon the identity in \cref{P:MassDiffBound}, that $m(\mu(t))$ remains zero. \cref{P:MassDiffBound} does show, however, that,  as $\nu \to 0$, $m(\mu)$ vanishes if $\norm{\mu}$ vanishes.  This will be very useful to us in \cref{S:VVNonL2}.  

\begin{theorem}\label{T:VVPAG}
	Let $T$ be as in \cref{T:ViscousExistence}.
	Assume that $\rho_0$ is compactly supported and in \Ignore{$W^{s,p}(\R^d)$ for $ s> 1$ and $sp>d$}$C^{\al}_C$ for some $\al > 1$.  Also assume
	$d \ge 3$ or $\sigma_1 + \sigma_2 = 0$.
	Then for all $\nu \le 1$ and $t \in [0, T]$,
	\begin{align*}
		\norm{(\v^\nu - \v^0)(t)}_{H^1}^2 + \norm{(\rho^\nu - \rho^0)(t)}^2
			+ \nu \int_0^t \norm{(\rho^\nu - \rho^0)(s)}^2 \, ds
			\le C_0(t) t \nu e^{C_0(t)}.
	\end{align*}
\end{theorem}
\begin{proof}
Define
\begin{align}\label{e:muwDefs}
	\begin{array}{ll}
	\mu = \rho^\nu - \rho^0,
			&\w = \v^\nu - \v^0.
	\end{array}
\end{align}
Then $\dv \w = \sigma_1 \mu$, and $\w \in L^2(\R^d)$ for all time by \cref{L:MassZero}, since $m(\mu) = 0$ by \cref{P:MassDiffBound}. 

Taking $(GAG_\nu) - (GAG_0)$ gives
\begin{align*}
	\prt_t \mu + \w \cdot \grad \rho^0 + \v^\nu \cdot \grad \mu
		= \sigma_2 \mu (\rho^0 + \rho^\nu) + \nu \Delta \rho^\nu.
\end{align*}
Noting that each term above is at least in $L^2(0, T; H^{-1})$ by \cref{T:ViscousExistence,T:InviscidExistence}, we can take the pairing of the above equation with $\varphi \in C_C^\iny([0, T) \times \R^d)$, giving
\begin{align}\label{e:mulDiffEqvarphi}
	&(\prt_t \mu, \varphi)
		= - \innp{\w \cdot \grad \rho^0, \varphi} - \innp{\v^\nu \cdot \grad \mu, \varphi}
			+ \sigma_2 \innp{\mu (\rho^0 + \rho^\nu), \varphi} + \nu (\Delta \rho^\nu, \varphi).
\end{align}

By density, \cref{e:mulDiffEqvarphi} holds as well if we set $\varphi = \mu \in L^2(0, T; H^1)$. Then,
\begin{align*}
	(\prt_t \mu, \varphi)
		&= \frac{1}{2} \diff{}{t} \norm{\mu}^2,
			\\
	-\innp{\v^\nu \cdot \grad \mu, \varphi}
		&= - \frac{1}{2} \innp{ \v^\nu, \grad \mu^2}
		= \frac{1}{2} \innp{\dv \v^\nu, \mu^2}
		= \frac{\sigma_1}{2} \innp{\rho^\nu, \mu^2}
		\le \frac{\sigma_1}{2} \norm{\rho^\nu}_{L^\iny} \norm{\mu}^2 \\
		&\le C_0(t) \norm{\mu}^2,
		\\
	\sigma_2 \innp{\mu (\rho^0 + \rho^\nu), \varphi}
		&\le \abs{\sigma_2} \norm{\rho^0 + \rho^\nu}_{L^\iny} \norm{\mu}^2
		\le C_0(t) \norm{\mu}^2,
		\\
	\nu (\Delta \rho^\nu, \varphi)
		&= - \nu \innp{\grad \rho^\nu, \grad \mu}
		= - \nu \innp{\grad \mu, \grad \mu} - \nu \innp{\grad \rho^0, \grad \mu} \\
		&\le - \nu \innp{\grad \mu, \grad \mu}
			+ \frac{\nu}{2} \norm{\grad \rho^0}^2
			+ \frac{\nu}{2} \norm{\grad \mu}^2
		\le C_0(t) \nu
			- \frac{\nu}{2} \norm{\grad \mu}^2.
\end{align*}
For the time derivative, we used Theorem 3 Section 5.9 of \cite{Evans} (or see\cite{T2001}). 

To estimate the term $-\innp{\w \cdot \grad \rho^0, \varphi}$, we consider the cases $d=2$ and $d=3$ separately.  Note that when $d=3$, by the Hardy-Littlewood-Sobolev inequality, $\| \w \|_{L^6} \leq C\| \mu \|$.  Therefore, when $d=3$,
\begin{equation*}
\begin{split}
&-\innp{\w \cdot \grad \rho^0, \mu} \leq \| \w \cdot \grad \rho^0 \| \| \mu\| \leq \| \w \|_{L^6} \| \grad \rho^0 \|_{L^3} \| \mu\|\\
&\qquad  \leq C\|\mu\|  \| \grad \rho^0 \|_{L^3} \| \mu\| = C\|\mu\|^2  \| \grad \rho^0 \|_{L^3}
	\le C_0(t) \norm{\mu}^2.
\end{split}
\end{equation*}
For the case $d=2$, the Hardy-Littlewood-Sobolev inequality does not yield the desired estimate, but we have 
\begin{equation}\label{term1}
\begin{split}
&-\innp{\w \cdot \grad \rho^0, \mu} \leq \| \w \cdot \grad \rho^0 \| \| \mu\|  \leq \| \w \|_{L^p} \| \grad \rho^0 \|_{L^q} \| \mu\|
\end{split}
\end{equation}
where $1\slash p + 1\slash q = 1\slash 2$.  Now note that for $p\in(2,\infty)$, by Bernstein's Lemma and boundedness of Calderon-Zygmund operators on $L^2$ (recall the definition of $\Delta_j$ in \cref{S:InfiniteEnergy}),
\begin{align}\label{e:wLpBound}
\begin{split}
&\| \w \|_{L^p} \leq \| \Delta_{-1} \w\|_{L^p} + \sum_{j\geq 0} \| \Delta_j \w \|_{L^p} \leq  C\| \Delta_{-1} \w\| + C \sum_{j\geq 0} 2^{-j}\| \Delta_j \nabla \nabla\Phi\ast \mu \|_{L^p}\\
&\qquad \leq C\| \w\| + C \sum_{j\geq 0} 2^{-j}2^{2j(1\slash 2 - 1\slash p)}\| \Delta_j \nabla \nabla\Phi\ast \mu \|\\
&\qquad \leq C\| \w\| + C \sum_{j\geq 0} 2^{-(2j \slash p)}\| \mu \| \leq C(\| \w\| + \| \mu \|).
\end{split}
\end{align} 
Substituting this estimate into (\ref{term1}) gives, for any fixed $q\in (2,\infty)$, 
\begin{equation}\label{VV2dterm}
-\innp{\w \cdot \grad \rho^0, \mu} \leq \| \w \cdot \grad \rho^0 \| \| \mu\| \leq C\| \grad \rho^0 \|_{L^q}( \| \w\| + \| \mu \| )\| \mu \|
	\le C_0(t) \pr{\norm{\w} + \norm{\mu}} \norm{\mu}.
\end{equation}
 
Applying the above estimates to \cref{e:mulDiffEqvarphi}, we see that for $d=3$,
\begin{align}\label{VV3d}
	\diff{}{t} \norm{\mu}^2 + \nu \norm{\grad \mu}^2
		&\le C_0(t) \nu +  C_0(t) \norm{\mu}^2,
\end{align}
while, for $d=2$,
\begin{align}\label{VV2d}
	\diff{}{t} \norm{\mu}^2 + \nu \norm{\grad \mu}^2
		&\le C_0(t) \nu + C_0(t) \norm{\w}^2 + C_0(t) \norm{\mu}^2.
\end{align}
After integrating (\ref{VV3d}) in time and applying Gronwall's Lemma, we can conclude that $\rho_{\nu}$ converges to $\rho$ in  $L^{\infty}([0,T]; L^2(\R^3))$ as $\nu$ approaches zero.  However, we must obtain a bound on $\norm{\w}$  for both $d=2$ and $d=3$ below in order to obtain the estimate in Theorem \ref{T:VVPAG} on the difference of velocities in $H^1$.  Therefore, in what follows, we utilize (\ref{VV2d}) for both $d=2$ and $d=3$. 
   
Integrating (\ref{VV2d}) in time, we have
\begin{align}\label{e:mulNormBound}
	\norm{\mu(t)}^2 + \nu \int_0^t \norm{\grad \mu(s)}^2 \, ds
		&\le C_0(t) t \nu + \int_0^t C_0(s) \pr{\norm{\w(s)}^2 + \norm{\mu(s)}^2} \, ds.
\end{align}
  
We return to \cref{e:mulDiffEqvarphi}, sticking for the moment with an unspecified $\varphi \in L^2(0, T; H^1)$, but integrating several of the terms by parts in a different manner than above. We have,
\begin{align}\label{e:IBPsVarious}
	\begin{split}
	(\prt_t \mu, \varphi)
		&= \sigma_1^{-1} (\prt_t \dv \w, \varphi)
			= - \sigma_1^{-1} \innp{\prt_t \w, \grad \varphi},
			\\
	-\innp{\w \cdot \grad \rho^0, \varphi}
		&= -\innp{\varphi \w, \grad \rho^0}
		= \innp{\dv (\varphi \w), \rho_0}
		= \innp{\varphi \dv \w, \rho^0} + \innp{\grad \varphi \cdot \w, \rho^0} \\
		&= \sigma_1 \innp{\varphi \mu, \rho^0} + \innp{\grad \varphi \cdot \w, \rho^0},
		\\
	-\innp{\v^\nu \cdot \grad \mu, \varphi}
		&= -\innp{\varphi \v^\nu, \grad \mu}
		= \innp{\dv(\varphi \v^\nu), \mu}
		= \innp{\varphi \dv \v^\nu, \mu} + \innp{\grad \varphi \cdot \v^\nu, \mu} \\
		&= \sigma_1 \innp{\varphi \rho^\nu, \mu} + \innp{\grad \varphi \cdot \v^\nu, \mu}, \\
	\nu (\Delta \rho^\nu, \varphi)
		&= - \nu \innp{\grad \rho^\nu, \grad \varphi}.
	\end{split}
\end{align}
From \cref{e:mulDiffEqvarphi}, then, we have
\begin{align}\label{e:mu1mu2AssumptionNeeded}
	\begin{split}
	- \sigma_1^{-1} &\innp{\prt_t \w, \grad \varphi}
		= \sigma_1 \innp{\varphi \mu, \rho^0} + \innp{\grad \varphi \cdot \w, \rho^0}
			+ \sigma_1 \innp{\varphi \rho^\nu, \mu} + \innp{\grad \varphi \cdot \v^\nu, \mu} \\
		&\qquad\qquad\qquad\qquad
			+ \sigma_2 \innp{\mu (\rho^0 + \rho^\nu), \varphi}
			- \nu \innp{\grad \rho^\nu, \grad \varphi} \\
		&= \innp{\grad \varphi \cdot \w, \rho^0}
			+ \innp{\grad \varphi \cdot \v^\nu, \mu}
			+ (\sigma_1 + \sigma_2) \pr{\mu (\rho^0 + \rho^\nu), \varphi}
			- \nu \innp{\grad \rho^\nu, \grad \varphi}.
	\end{split}
\end{align}

Now set $\varphi = \sigma_1\F * \mu$. For $d \ge 3$, $\varphi$ lies in $L^2(0, T; \dot{H}^1)$ by \cref{T:ViscousExistence,L:MassZero}, so \cref{e:mu1mu2AssumptionNeeded} continues to hold by the density of $C_C^\iny([0, T) \times \R^d)$ in $L^2(0, T; \dot{H}^1)$. The same is true for $d = 2$ when $\sigma_1 + \sigma_2 = 0$  (though then the only term on the right-hand side of \cref{e:mu1mu2AssumptionNeeded} involving $\varphi$ without a gradient vanishes).
Hence, in both cases covered by this theorem, \cref{e:mu1mu2AssumptionNeeded}  holds for $\varphi = \sigma_1\F * \mu$.
Then $\grad \varphi = \w$, and we find that
\begin{align}\label{e:mulDiffEqvarphiSimp}
	\begin{split}
	\sigma_1^{-1} &\innp{\prt_t \w, \w}
		= \frac{\sigma_1^{-1}}{2} \diff{}{t} \norm{\w}^2 \\
		&= -\innp{\abs{\w}^2, \rho^0}
			- \innp{\w \cdot \v^\nu, \mu}
			+ \nu \innp{\grad \rho^\nu, \w}
			-\sigma_1 (\sigma_1 + \sigma_2)
				\innp{\mu \eta, \F * \mu},
	\end{split}
\end{align}
where $\eta := \rho^0 + \rho^\nu$.
But,
\begin{align*}
	\abs{\innp{\abs{\w}^2, \rho^0}}
		&\le \norm{\rho^0}_{L^\iny} \norm{\w}^2
		\le C_0(t) \norm{\w}^2,
		\\
	\abs{\innp{\w \cdot \v^\nu, \mu}}
		&\le \norm{\v^\nu}_{L^\iny} \norm{\w} \norm{\mu}
		\le C_0(t) \norm{\w}^2 + C_0(t) \norm{\mu}^2,
		\\
	\nu \abs{\innp{\grad \rho^\nu, \w}}
		&= \nu \abs{\innp{\grad \mu, \w}
			+ \innp{\grad \rho^0, \w}}
		\le \frac{\nu \abs{\sigma_1^{-1}}}{4} \norm{\grad \mu}^2
			+ \frac{2\nu + \abs{\sigma_1^{-1}}}{2 \abs{\sigma_1^{-1}}} \norm{\w}^2
			+ \frac{\nu^2}{2} \norm{\grad \rho^0}^2 \\
		&\le \frac{\nu \abs{\sigma_1^{-1}}}{4} \norm{\grad \mu}^2
			+ \frac{2\nu + \abs{\sigma_1^{-1}}}{2 \abs{\sigma_1^{-1}}} \norm{\w}^2
			+ C_0(t) \nu^2.
\end{align*}

Now consider $\langle \mu\eta, \Phi*\mu\rangle = \langle \mu\eta, \nabla\Phi*\textbf{w}\rangle$ for $d=3$.  Write
\begin{equation*}
\begin{split}
&\abs{\langle \mu\eta, \nabla\Phi*\textbf{w}\rangle} \leq \|\mu\eta \|_{L^{6\slash 5}} \| \nabla\Phi*\textbf{w} \|_{L^6}\\ 
&\qquad\leq \| \mu \|_{L^2} \| \eta \|_{L^3} \| \nabla\Phi*\textbf{w} \|_{L^6} \leq C_0(t)\| \mu \|_{L^2} \|\textbf{w} \|_{L^2},
\end{split}
\end{equation*}
where we applied the Hardy-Littlewood-Sobolev Inequality.
When $\sigma_1 = - \sigma_2$, the term $\sigma_1 (\sigma_1 + \sigma_2) \innp{\mu \eta, \F * \mu}$ disappears entirely.

Substituting these bounds into \cref{e:mulDiffEqvarphiSimp} and integrating in time gives, for all $\nu \le 1$,
\begin{align*} 
	\norm{\w(t)}^2
		&\le \frac{\nu}{2} \int_0^t \norm{\grad \mu(s)}^2 \, ds+ C_0(t) t \nu
			+ \int_0^t C_0(s) \pr{\norm{\w(s)}^2 + \norm{\mu(s)}^2} \, ds.
\end{align*}
Adding this inequality to that in \cref{e:mulNormBound} gives, for all $\nu \le 1$,
\begin{align*} 
	\norm{\w(t)}^2 + \norm{\mu(t)}^2 + \frac{\nu}{2} \int_0^t \norm{\grad \mu(s)}^2 \, ds
		&\le C_0(t) t \nu + \int_0^t C_0(s) \pr{\norm{\w(s)}^2 + \norm{\mu(s)}^2} \, ds.
\end{align*}
Applying Gronwall's lemma, we conclude that
\begin{align*}
	\norm{\w(t)}^2 + \norm{\mu(t)}^2 + \nu \int_0^t \norm{\grad \mu(s)}^2 \, ds
		\le C_0(t) t \nu e^{C_0(t)}
\end{align*}
for all $\nu \le 1$. The proof is completed by observing that $\norm{\w(t)}_{H^1} \le \norm{\w(t)} + C \norm{\mu(t)}$ by the boundedness of Calderon-Zygmund operators on $L^2$.
\end{proof}

\begin{remark}
An examination of the proof of \cref{T:VVPAG} shows that the conclusion holds as long as the solutions satisfy $\nabla\rho^0\in L^{\infty}([0,T], L^2\cap L^3(\R^d))$ when $d=3$, and $\nabla\rho^0\in L^{\infty}([0,T], L^2\cap L^q(\R^d))$ for some $q>2$ when $d=2$, and $\rho^{\nu}$, $\rho^{0}$ belong to $L^{\infty}([0,T], L^1\cap L^{\infty}(\R^d))$ for $d=2$, $3$. Our assumptions on the initial data imply that these conditions hold, but are not minimal; in particular, compact support of $\rho_0$ is stronger than required.
\end{remark}

%
%
\section{The vanishing viscosity limit for \GAGnu for velocities not in $L^2$}\label{S:VVNonL2}

\noindent In this section, we consider the vanishing viscosity limit $\VV'$ (see \cref{S:Introduction}) in the general 2D case.

\begin{theorem}\label{T:VVGAG}
	Assume that $d = 2$ and $\rho_0$ is compactly supported and in $C^{\al}_C$ for some $\al > 1$\Ignore{$W^{s,p}(\R^2)$ for some $s > 1$ with $sp>2$}.
	Define $\mu$, $\w$ as in \cref{e:muwDefs}.
	Let $T$ be as in \cref{T:ViscousExistence}.
	With $\tau_0$, $g_0$ as in \cref{D:btau0} and 
	the total mass function, $m$, defined as in \cref{e:MassDef},
	let
	\begin{align}\label{e:mulwlDef}
		\begin{array}{lll}
			\mul := \mu - m(\mu) g_0,
				&\wl := \w - \sigma_1 m(\mu) \btau_0.
		\end{array}
	\end{align}
	Then for all $t \in [0, T]$, $\nu \le 1$,
	\begin{align}\label{e:FinalL2Bound}
			\norm{\wl(t)}_{H^1}^2 + \norm{\mu(t)}^2
				+ \nu \int_0^t \norm{\grad \mu}^2
			\le C \nu t e^{C_0(t) t}.
	\end{align}
	Moreover, for all $k \ge 0$,
	\begin{align}\label{e:wInfDiff}
		\norm{\w(t) - \wl(t)}_{C^k}
			\le C_k \nu^{\frac{1}{2}} t^{\frac{3}{2}} e^{C_0(t) t}.
	\end{align}
\end{theorem}

\begin{proof}
Assume that $d = 2$.
Then $\dv \wl = \sigma_1 \mul$, and $\wl \in L^2(\R^d)$ for all time by \cref{L:MassZero}, since $m(\mul) = 0$. 

We start off the same way as in \cref{S:VV}. Taking $(GAG_\nu) - (GAG_0)$ gives
\begin{align*}
	\prt_t \mu + \w \cdot \grad \rho^0 + \v^\nu \cdot \grad \mu
		= \sigma_2 \mu (\rho^0 + \rho^\nu) + \nu \Delta \rho^\nu.
\end{align*}
We have $\prt_t \mu, \Delta \rho^\nu \in L^2(0, T; H^{-1})$ while all other terms above lie in the smaller space, $L^\iny(0, T; L^2)$. Thus, we can take the pairing of the above equation with $\varphi = \mu \in L^2(0, T; H^1)$, giving
\begin{align}\label{e:mulDiffEqvarphi2}
	&(\prt_t \mu, \mu)
		= - \innp{\w \cdot \grad \rho^0, \mu} - \innp{\v^\nu \cdot \grad \mu, \mu}
			+ \sigma_2 \innp{\mu (\rho^0 + \rho^\nu), \mu} + \nu (\Delta \rho^\nu, \mu),
\end{align}
as in \cref{S:VV}, immediately following \cref{e:mulDiffEqvarphi}. Now, however, we need to estimate $- \innp{\w \cdot \grad \rho^0, \mu}$ differently.  First note that
\begin{align}\label{e:wgradrho0muBound}
		\abs{\innp{\w \cdot \grad \rho^0, \mu}}
		&\le \abs{\innp{\wl \cdot \grad \rho^0, \mu}}
			+ \abs{\innp{(\sigma_1 m(\mu) \btau_0) \cdot \grad \rho^0, \mu}}.
\end{align}
Following the proof of Theorem \ref{T:VVPAG}, we have  
\begin{equation}\label{term12d}
\begin{split}
&|\innp{\wl \cdot \grad \rho^0, \mu}| \leq \| \wl \cdot \grad \rho^0 \| \| \mu\| \leq \| \wl \|_{L^p} \| \grad \rho^0 \|_{L^q} \| \mu\|
\end{split}
\end{equation}
where $1\slash p + 1\slash q = 1\slash 2$.  Now note that for $p\in(2,\infty)$, an argument identical to that in \cref{e:wLpBound} yields
\begin{align*}
	\norm{\wl}_{L^p}
		\le C \pr{\norm{\wl} + \norm{\mul}}.
\end{align*}
Substituting this estimate into \cref{term12d} gives, for any fixed $q\in (2,\infty)$, 
\begin{equation*}
	|\innp{\wl \cdot \grad \rho^0, \mu}| 
		\leq C\| \grad \rho^0 \|_{L^q}
			( \| \wl\| + \| \mul \| )\| \mu \|
		\le C_0(t) \pr{\norm{\wl} + \norm{\mul}}
				\norm{\mu}.
\end{equation*}
Now, by the definition of $\mul$ and Propostion \ref{P:MassDiffBound},
\begin{align*}
	 &\norm{\mul(t)}
	 	\le  \norm{\mu(t)} + \abs{m(\mu(t))} \norm{g_0}
		\le \norm{\mu(t)} + \norm{g_0}
			\abs{\sigma_1 + \sigma_2}
		\int_0^t \norm{\mu(s)} \norm{\rho^{\nu}
			+\rho^0(s)} \, ds\\
		&\qquad
		\le  \norm{\mu(t)} + C_0(t)
			\int_0^t \norm{\mu(s)} \, ds,
\end{align*}		
so that 
\begin{align}\label{e:wrhomuInnBound}
\begin{split}
&|\innp{\wl \cdot \grad \rho^0, \mu}|  \leq C_0(t) \| \wl \| \| \mu \| + C_0(t)\|\mu \|^2 +  C_0(t)  \| \mu \|\int_0^t \norm{\mu(s)} \, ds\\
&\qquad \leq C_0(t)\| \wl \|^2 + C_0(t)  \| \mu \|^2 + \left(\int_0^t \norm{\mu(s)} \, ds \right)^2\\
&\qquad \leq C_0(t)\| \wl \|^2 + C_0(t)  \| \mu \|^2 + C_0(t)t\int_0^t \norm{\mu(s)}^2 \, ds,
\end{split}
\end{align}
where we used Jensen's inequality (or Cauchy-Schwartz) in the last step.
To estimate $\abs{((\sigma_1 m(\mu) \btau_0) \cdot \grad \rho^0, \mu)}$, we observe that, by \cref{P:MassDiffBound},
\begin{align}\label{e:mulBoundOrig}
		\abs{m(\mu(t))}
			&\le \abs{\sigma_1 + \sigma_2} \int_0^t \pr{\norm{\rho^\nu(s)} + \norm{\rho^0(s)}}
					\norm{\mu(s)} \, ds
			\le C_0(t) \int_0^t \norm{\mu(s)} \, ds
\end{align}
so that
\begin{align}\label{e:sigmammulbtau0Orig}
	\begin{split}
	&\abs{\innp{(\sigma_1 m(\mu) \btau_0) \cdot \grad \rho^0, \mu}}
		\le \abs{\sigma_1} \norm{\grad \rho^0} \norm{\btau_0}_{L^\iny}
				\abs{m(\mu)} \norm{\mu} \\
		&\qquad
		\le \frac{1}{2} \sigma_1^2 \norm{\grad \rho^0}^2
				\norm{\btau_0}_{L^\iny}^2 m(\mu)^2
				+ \frac{1}{2} \norm{\mu}^2
		\le C_0(t) \pr{\int_0^t \norm{\mu(s)} \, ds}^2
				+ \frac{1}{2} \norm{\mu}^2 \\
		&\qquad
		\le C_0(t) t
				\int_0^t \norm{\mu(s)}^2 \, ds
				+ \frac{1}{2} \norm{\mu}^2,
	\end{split}
\end{align}
where we again used Jensen's inequality (or Cauchy-Schwartz) as well as \cref{P:MassDiffBound}.
Hence, applying \cref{e:wrhomuInnBound,e:sigmammulbtau0Orig} to \cref{e:wgradrho0muBound},
\begin{align*}
	\abs{\innp{\w \cdot \grad \rho^0, \mu}}
		\le C_0(t) \norm{\wl}^2 + C_0(t) \norm{\mu}^2
			+ C_0(t) t \int_0^t \norm{\mu(s)}^2 \, ds.
\end{align*}

Integrating in time, we have
\begin{align}\label{e:inttinttBoundOrig}
	\int_0^t C_0(y) y \int_0^y \norm{\mu(s)}^2 \, ds \, dy
		\le C_0(t) t \int_0^t \int_0^t \norm{\mu(s)}^2 \, ds \, dy
		\le C_0(t) t^2 \int_0^t \norm{\mu(s)}^2 \, ds.
\end{align}
Thus, in place of \cref{e:mulNormBound}, we have
\begin{align}\label{e:mulNormBound2}
	\norm{\mu(t)}^2 + \nu \int_0^t \norm{\grad \mu(s)}^2 \, ds
		&\le C_0(t) t \nu + \int_0^t C_0(s) \pr{\norm{\wl(s)}^2
			+ \norm{\mu(s)}^2} \, ds.
\end{align}

To bound $\norm{\wl}$, we derive the equivalent of \cref{e:mu1mu2AssumptionNeeded} for $\wl$ for an arbitrary $\varphi \in L^2(0, T; H^1)$. The only change we need make is in $\cref{e:IBPsVarious}_1$, which, using $\dv \w = \dv \wl + \sigma_1 m(\mu) g_0$, becomes
\begin{align*}
	(\prt_t \mu, \varphi)
		&= \sigma_1^{-1} (\prt_t \dv \w, \varphi)
		= \sigma_1^{-1} (\prt_t \dv \wl, \varphi)
			+ m'(\mu) (g_0, \varphi) \\
		&= - \sigma_1^{-1} \innp{\prt_t \wl, \grad \varphi}
			+ (\sigma_1 + \sigma_2) m(\mu \eta) (g_0, \varphi),
\end{align*}
where $\eta := \rho^0 + \rho^\nu$ and we used \cref{P:MassDiffBound}.

Thus, \cref{e:mu1mu2AssumptionNeeded} becomes
\begin{align}\label{e:mu1mu2AssumptionNeededAlt}
	\begin{split}
	- \sigma_1^{-1} &\innp{\prt_t \wl, \grad \varphi}
		= \innp{\grad \varphi \cdot \w, \rho^0}
			+ \innp{\grad \varphi \cdot \v^\nu, \mu}
			+ (\sigma_1 + \sigma_2) \innp{\mu \eta, \varphi} \\
		&\qquad\qquad\qquad\qquad
			- \nu \innp{\grad \rho^\nu, \grad \varphi}
			- (\sigma_1 + \sigma_2) m(\mu \eta) (g_0, \varphi) \\
		&= \innp{\grad \varphi \cdot \w, \rho^0}
			+ \innp{\grad \varphi \cdot \v^\nu, \mu}
			- \nu \innp{\grad \rho^\nu, \grad \varphi}
			+ (\sigma_1 + \sigma_2) \innp{\mu \eta - m(\mu \eta) g_0, \varphi}.
	\end{split}
\end{align}

Now set $\varphi = \sigma_1\F * \mul = \sigma_1\F * (\mu - m(\mu) g_0)$ and fix $p_0 \in (2, \iny]$. Then $\grad \varphi \in L^\iny(0, T; L^2)$ with $\varphi \in L^\iny(0, T; L^{p_0})$ by \cref{L:MassZero}. Also, $\mu \eta \in L^\iny(0, T; L^{q_0})$, where $q_0$ is \Holder conjugate to $p_0$, by the rapid spatial decay of $\mu$ and $\eta$. Hence, \cref{e:mu1mu2AssumptionNeededAlt} holds for $\varphi$ by an obvious density argument. Hence, using $\grad \varphi = \wl$,
\begin{align}\label{e:mulDiffEqvarphiSimpAlt}
	\begin{split}
	\sigma_1^{-1} \innp{\prt_t \wl, \wl}
		&= \frac{\sigma_1^{-1}}{2} \diff{}{t} \norm{\wl}^2
		= -\innp{\wl \cdot \w, \rho^0}
			- \innp{\wl \cdot \v^\nu, \mu}
			+ \nu \innp{\grad \rho^\nu, \wl} \\
		&\qquad
			-\sigma_1
			(\sigma_1 + \sigma_2) \innp{\mu \eta - m(\mu \eta) g_0, \F * (\mu - m(\mu) g_0)}.
	\end{split}
\end{align}

Now,
\begin{align*}
	\abs{\innp{\wl \cdot \w, \rho^0}}
		&\le \norm{\rho^0}_{L^\iny} \norm{\wl}^2
			+  \abs{\sigma_1} \abs{m(\mu)} \abs{\innp{\btau_0 \cdot \wl, \rho_0}}
		\le C_0(t) \norm{\wl}^2 + C_0(t) \abs{m(\mu)}^2,
		\\
	\abs{\innp{\wl \cdot \v^\nu, \mu}}
		&\le \norm{\v^\nu}_{L^\iny} \norm{\wl} \norm{\mu}
		\le C_0(t) \norm{\wl}^2 + C_0(t) \norm{\mu}^2,
		\\
	\nu \abs{\innp{\grad \rho^\nu, \wl}}
		&= \nu \abs{\innp{\grad \mu, \wl}
			+ \innp{\grad \rho^0, \wl}}
		\le \frac{\nu \abs{\sigma_1^{-1}}}{4} \norm{\grad \mu}^2
			+ \frac{2 \nu + \abs{\sigma_1^{-1}}}{2 \abs{\sigma_1^{-1}}} \norm{\wl}^2
			+ \frac{\nu^2}{2} \norm{\grad \rho^0}^2 \\
		&\le \frac{\nu \abs{\sigma_1^{-1}}}{4} \norm{\grad \mu}^2
			+ \frac{2 \nu + \abs{\sigma_1^{-1}}}{2 \abs{\sigma_1^{-1}}} \norm{\wl}^2
			+ C_0(t) \nu^2.
\end{align*}

Substituting these bounds into \cref{e:mulDiffEqvarphiSimpAlt}, integrating in time, using \cref{e:mulBoundOrig} with Jensen's inequality, and applying \cref{L:muetaBound} we obtain, for all $\nu \le 1$,
\begin{align*} 
	\norm{\wl(t)}^2
		&\le \int_0^t\frac{\nu}{2} \norm{\grad \mu(s)}^2 \, ds+ C_0(t) t \nu
			+ \int_0^t C_0(s) \pr{\norm{\wl(s)}^2 + \norm{\mu(s)}^2} \, ds.
\end{align*}
Adding this inequality to that in \cref{e:mulNormBound2} gives, for all $\nu \le 1$,
\begin{align*} 
	\norm{\wl(t)}^2 + \norm{\mu(t)}^2 + \frac{\nu}{2} \int_0^t \norm{\grad \mu(s)}^2 \, ds
		&\le C_0(t) t \nu + \int_0^t C_0(s) \pr{\norm{\wl(s)}^2 + \norm{\mu(s)}^2} \, ds.
\end{align*}
Applying Gronwall's lemma, we conclude that
\begin{align}\label{e:wlmuBoundPre}
	\norm{\wl(t)}^2 + \norm{\mu(t)}^2 + \nu \int_0^t \norm{\grad \mu(s)}^2 \, ds
		\le C_0(t) t \nu e^{C_0(t)}
\end{align}
for all $\nu \le 1$. Also, by \cref{L:MassZero,P:MassDiffBound},
\begin{align*}
	\norm{\wl(t)}_{H^1}
		&\le \norm{\wl(t)} + C \norm{\mul(t)}
		\le \norm{\wl(t)} + \norm{\mu(t)} + \abs{m(\mu(t))} \norm{g_0}  \\
		&\le \norm{\wl(t)} + \norm{\mu(t)}
			+ \norm{g_0} \abs{\sigma_1 + \sigma_2} \int_0^t \norm{\mu(s)} \norm{\eta(s)} \, ds\\
		&\le \norm{\wl(t)} + \norm{\mu(t)} + C_0(t) \int_0^t \norm{\mu(s)} \, ds \\
		&\le \norm{\wl(t)} + C_0(t) {(t \nu)}^{\frac{1}{2}} e^{C_0(t)} + C_0(t) \int_0^t C_0(s) s^{\frac{1}{2}} {\nu}^{\frac{1}{2}} e^{C_0(s)} \, ds\\
		&\le \norm{\wl(t)} + C_0(t) (1+t)^{\frac{3}{2}} e^{C_0(t)} {\nu}^{\frac{1}{2}}.
\end{align*}
In the second-to-last inequality, we used \cref{e:wlmuBoundPre}. Combining this bound with \cref{e:wlmuBoundPre} completes the proof of \cref{e:FinalL2Bound}.

To prove \cref{e:wInfDiff}, we simply observe that
\begin{align*}
	\norm{\w(t) - \wl(t)}_{L^\iny}
		&= \norm{\sigma_1 m(\mu) \btau_0}_{L^\iny}
		\le \abs{\sigma_1} \abs{m(\mu)} \norm{\btau_0}_{L^\iny}
		\le C \abs{m(\mu)} \\
		&\le C_0(t) \int_0^t \norm{\mu(s)} \norm{\eta(s)} \, ds
		\le C \nu^{\frac{1}{2}} t^{\frac{3}{2}} e^{C_0(t) t},
\end{align*}
where we used \cref{P:MassDiffBound} and \cref{e:FinalL2Bound}. A similar bound holds for all spatial derivatives of $\w(t) - \wl(t)$, yielding \cref{e:wInfDiff}.
\end{proof}

We used the following lemma in the proof of \cref{T:VVGAG}, above.

\begin{lemma}\label{L:muetaBound}
	Define $\mu$,
	$\mul$, $\wl$ as in \cref{e:muwDefs,e:mulwlDef}
	and let $\eta = \rho^0 + \rho^\nu$, as in the proof of
	\cref{T:VVGAG}.
	When $d = 2$, we have,
	\begin{align*}
		\abs{\innp{\mu \eta - m(\mu \eta) g_0, \F * \mul}}
			\le C_0(t) \norm{\wl} \norm{\mu}.
	\end{align*}
\end{lemma}
\begin{proof}
	First observe that $\gamma := \mu \eta - m(\mu \eta) g_0$ lies in $L^2(\R^2)$,
	has total mass zero, and
	$|x|^{\epsilon} \gamma \in L^1(\R^2)$ for some $\epsilon\in(0,1)$. Thus, by \cref{L:MassZero},
	$\grad \F * \gamma \in L^2(\R^2)$. Similarly, $\mul \in L^2(\R^2)$
	and $\grad \F * \mul = \wl \in L^2(\R^2)$. 
	This allows us to integrate by parts, using $\gamma = \dv (\grad \F * \gamma))$, to conclude that
	\begin{align*}
		\abs{\innp{\gamma, \F * \mul}}
			&= \abs{\innp{\grad \F * \gamma, \wl}}
			\le \norm{\grad \F * \gamma} \norm{\wl}.
	\end{align*}
	
	By \cref{L:NearL2Norm}, with $a$ as in \cref{D:Radial},
	\begin{align}\label{e:gradPhigamma}
		\begin{split}
		\norm{\grad \F * \gamma}^2
			&= - \innp{\F * \gamma, \gamma}
			= - \innp{(a \F) * \gamma, \gamma}
				- \innp{((1 - a) \F) * \gamma, \gamma} \\
			&\le \norm{a \F}_{L^1} \norm{\gamma}^2 + \abs{\innp{((1 - a) \F) * \gamma, \gamma}} \\
			&\le C_0(t) \norm{\mu}^2 + \abs{\innp{((1 - a) \F) * \gamma, \gamma}},
	\end{split}
	\end{align}
	since
	\begin{align*}
		\norm{\gamma}^2
			&= \norm{\mu \eta - m(\mu \eta) g_0}^2
				\le \pr{\norm{\mu} \norm{\eta}_{L^\iny} + \abs{m(\mu \eta)} \norm{g_0}}^2 \\
			&\le \pr{\norm{\mu} \norm{\eta}_{L^\iny} + \norm{\mu} \norm{\eta} \norm{g_0}}^2
			\le C_0(t) \norm{\mu}^2.
	\end{align*}
	
	It remains to estimate $\abs{\innp{((1 - a) \F) * \gamma, \gamma}}$.
	Define $g(x) := 1 + \abs{x}^{\epsilon}$ and write,
	\begin{align*}
		\begin{split}
		&\abs{\innp{((1 - a) \F) * \gamma, \gamma}}
			= \abs{\innp{(1/g) \brac{((1 - a) \F) * \gamma}, g \gamma}} \\
			&\qquad
			\le \norm{(1/g) \brac{((1 - a) \F) * \gamma}}_{L^\iny} \norm{g \gamma}_{L^1}.
	\end{split}
	\end{align*}
	The compact support of $\rho_0$ insures the
	compact support of $\rho^0$ and, through
	\cref{C:L1Membership}, the rapid spatial decay of
	$\rho^\nu$. This allows us to conclude
	that $\norm{g \eta} \le C_0(t)$. Therefore,	
	\begin{align*}
		\norm{g \gamma}_{L^1}
			&= \norm{\mu g \eta - m(\mu \eta) g g_0)}_{L^1}
			\le \norm{\mu} \norm{g \eta} + \abs{m(\mu \eta)} \norm{g g_0}_{L^1} \\
			&\le C_0(t) \norm{\mu} + \norm{\mu} \norm{\eta} \norm{ g g_0}_{L^1}
			\le C_0(t) \norm{\mu},
	\end{align*}
	and we conclude that
	\begin{align}\label{e:OneaPhigamma}
		\begin{split}
		&\abs{\innp{((1 - a) \F) * \gamma, \gamma}}
			\le C_0(t) \norm{(1/g) \brac{((1 - a) \F) * \gamma}}_{L^\iny} \norm{\mu}.
	\end{split}
	\end{align}
	
	We need to extract another factor of $\norm{\mu}$ from
	$\norm{(1/g) \brac{((1 - a) \F) * \gamma}}_{L^\iny}$. We have,
	\begin{align*}
		\bigabs{\frac{1}{g(x)} &\brac{((1 - a) \F) * \gamma}(x)}
			= \bigabs{\frac{1}{2 \pi g(x)}
				\int_{\R^2} (1 - a(x - y)) \log \abs{x - y} \gamma(y) \, dy} \\
			&\le \frac{1}{2 \pi g(x)}
				\int_{\R^2} \frac{(1 - a(x - y)) \log \abs{x - y}}{\abs{x - y}^{\epsilon}}
						\abs{x - y}^{\epsilon} \abs{\gamma(y)} \, dy \\
			&\le \frac{1}{2 \pi g(x)}
				\int_{\R^2} \frac{(1 - a(x - y)) \log \abs{x - y}}{\abs{x - y}^{\epsilon}}
						(\abs{x} + \abs{y})^{\epsilon} \abs{\gamma(y)} \, dy \\
			&\le \frac{1}{2 \pi g(x)}
				\int_{\R^2} \frac{(1 - a(x - y)) \log \abs{x - y}}{\abs{x - y}^{\epsilon}}
						\abs{x}^{\epsilon} \abs{\gamma(y)} \, dy \\
			&\qquad
				+ \frac{1}{2 \pi g(x)}
				\int_{\R^2} \frac{(1 - a(x - y)) \log \abs{x - y}}{\abs{x - y}^{\epsilon}}
						\abs{y}^{\epsilon} \abs{\gamma(y)} \, dy.
	\end{align*} 
	So
	\begin{align}\label{e:OneovergConv}
		\begin{split}
		&\norm{\frac{1}{g} \brac{((1 - a) \F) * \gamma}}_{L^\iny}
			\le \norm{\frac{\abs{x}^{\epsilon}}{2 \pi g(x)}
				\int_{\R^2} \frac{(1 - a(x - y)) \log \abs{x - y}}{\abs{x - y}^{\epsilon}}
						\abs{\gamma(y)} \, dy}_{L^\iny_x} \\
			&\qquad\qquad
				+ \norm{\frac{1}{2 \pi g(x)}
				\int_{\R^2} \frac{(1 - a(x - y)) \log \abs{x - y}}{\abs{x - y}^{\epsilon}}
						\abs{y}^{\epsilon} \abs{\gamma(y)} \, dy}_{L^\iny_x} \\
			&\qquad
			\leq C \int_{\R^2} \abs{\gamma(y)} \, dy
				+ C \int_{\R^2} \abs{y}^{\epsilon} \abs{\gamma(y)} \, dy
			= C \norm{\gamma}_{L^1} + C \norm{\abs{x}^{\epsilon} \gamma(x)}_{L^1_x}.
	\end{split}
	\end{align}
	But,
	\begin{align*}
		\norm{\gamma}_{L^1}
			&= \norm{\mu \eta - m(\mu \eta) g_0}_{L^1}
			\le \norm{\mu} \norm{\eta} + \abs{m(\mu \eta)} \norm{g_0}_{L^1} \\
			&\le C_0(t) \norm{\mu} + \norm{\mu} \norm{\eta} \norm{g_0}_{L^1}
			\le C_0(t) \norm{\mu}
	\end{align*}
	and
	\begin{align*}
		\norm{\abs{x}^{\epsilon} \gamma(x)}_{L^1_x}
			&= \norm{\mu(x) (\abs{x}^{\epsilon} \eta(x)) - m(\mu \eta) (\abs{x}^{\epsilon} g_0(x))}_{L^1_x} \\
			&\le \norm{\mu} \norm{\abs{x}^{\epsilon} \eta(x)}_{L^2_x}
					+ \abs{m(\mu \eta)} \norm{\abs{x}^{\epsilon} g_0(x)}_{L^1_x} \\
			&\le C_0(t) \norm{\mu} + \norm{\mu} \norm{\eta} \norm{\abs{x}^{\epsilon} g_0(x)}_{L^1_x}
			\le C_0(t) \norm{\mu}.
	\end{align*}
	Substituting this estimate into \cref{e:OneovergConv},
	the resulting estimate into \cref{e:OneaPhigamma},
	and finally that estimate into
	\cref{e:gradPhigamma}, yields the desired bound.
\end{proof}

%
%
\section{The vanishing viscosity limit in the $L^{\infty}$-norm}\label{S:VVUniform}

In this section, we use the results from \cref{S:VV} and \cref{S:VVNonL2} to prove that the vanishing viscosity limit (see $(VV)$ in \cref{S:Introduction}) holds in the $L^{\infty}$-norm of the density. This follows immediately from interpolation if we are able to find a modulus of continuity on $\rho^\nu$ that applies for all sufficiently small $\nu$. We do this by obtaining a bound on $\rho^\nu$ in a \Holder space norm \textit{uniformly} in $\nu$, adapting the approach of Hmidi and Keraani in \cite{Hmidi2005,HK2008} for the Navier-Stokes equations. This leads to the following theorem.

\begin{theorem}\label{T:HolderUniform}
Assume $\rho_0\in C^{\beta}(\R^d)$ with $\beta<1$ and $d\geq 2$, and let $T$ be as in \cref{T:ViscousExistence}.  Then the smooth solution $\rho^{\nu}$ to ($GAG_{\nu}$) belongs in $L^{\infty}([0,T];C^{\beta}(\R^d))$, and the following estimate holds for all $t\in[0,T]$:
\begin{equation}\label{holderuniform}
\| \rho^{\nu}(t) \|_{C^{\beta}} \leq C_0(T)e^{e^{C_0(T)}}.
\end{equation}
where $C_0(T)$ depends on $\| \rho_0 \|_{C^{\beta}}$ and $\beta$.
\end{theorem}
\begin{proof}
That the solution $\rho^{\nu}$ to ($GAG_{\nu}$) belongs in $L^{\infty}([0,T];C^{\beta}(\R^d))$ {\em for all} $\beta>0$ follows from standard arguments; we show only the uniform control in viscosity of the $C^{\beta}$-norm when $\beta<1$.

This theorem is essentially proved in \cite{Hmidi2005,HK2008} for divergence-free vector fields $\v^{\nu}$ in the more general setting of Besov spaces.  We follow the proof from \cite{HK2008} below, with a slight modification to account for the assumption that $\v^{\nu}$ is not divergence-free in our setting.  Our modification relies on a commutator estimate established in Chapter 4 of \cite{C1998}.

As in the proof in \cite{HK2008}, we apply the Littlewood-Paley operator $\Delta_j$ to ($GAG_{\nu}$) and apply the maximum principle (see, for example, Lemma 5 of \cite{Hmidi2005}) to write
\begin{equation*} 
\| \Delta_j\rho^{\nu} (t) \|_{L^{\infty}} \leq \| \Delta_j\rho_0 \|_{L^{\infty}} + C\int_0^t \left(\| [\Delta_j, \v^{\nu}\cdot\nabla]\rho^{\nu}(s) \|_{L^{\infty}} + \| \Delta_j (\rho^{\nu}(s))^2\|_{L^{\infty}}\right) \, ds. 
\end{equation*}
Multiplying through by $2^{j\beta}$ and taking the supremum over $j$ gives
\begin{equation}\label{preGronwall}
\begin{split}
&\| \rho^{\nu} (t) \|_{C^{\beta}} \leq \| \rho_0 \|_{C^{\beta}} + C\int_0^t \left(\sup_j 2^{j\beta}\| [\Delta_j, \v^{\nu}\cdot\nabla]\rho^{\nu}(s) \|_{L^{\infty}} + \| (\rho^{\nu}(s))^2\|_{C^{\beta}}\right) \, ds.\\
&\qquad \leq  \| \rho_0 \|_{C^{\beta}} + C\int_0^t \left(\sup_j 2^{j\beta}\| [\Delta_j, \v^{\nu}\cdot\nabla]\rho^{\nu}(s) \|_{L^{\infty}} + \|\rho^{\nu}(s)\|_{L^{\infty}}\| \rho^{\nu}(s)\|_{C^{\beta}}\right) \, ds,
\end{split}
\end{equation}
where we used the estimate $\| (\rho^{\nu})^2\|_{C^{\beta}} \leq C\|\rho^{\nu}\|_{L^{\infty}}\| \rho^{\nu}\|_{C^{\beta}}$ to obtain the last inequality.  To bound the commutator on the right hand side in (\ref{preGronwall}), we apply the commutator estimate  
\begin{equation}\label{commest}
\| [\Delta_j, \v^{\nu}\cdot\nabla]\rho^{\nu}(s) \|_{L^{\infty}} \leq C2^{-j\beta} \| \nabla \v^{\nu}(s) \|_{L^{\infty}}\| \rho^{\nu}(s)\|_{C^{\beta}},
\end{equation}
which is established in Chemin's proof of Lemma 4.1.1 of \cite{C1998} (note that Chemin shows that (\ref{commest}) holds even when div $\v^{\nu} \neq 0$ and for \textit{all} $\beta>0$).  Substituting (\ref{commest}) into (\ref{preGronwall}) and applying Gronwall's lemma gives
\begin{equation}\label{mainholderboundest}
\| \rho^{\nu} \|_{L^{\infty}([0,T]; C^{\beta})} \leq Ce^{CV(t)} \| \rho_0\|_{C^{\beta}},
\end{equation}
where
\begin{equation*}
V(t) = \int_0^t \left(\| \nabla \v^{\nu} (s)\|_{L^{\infty}}+ \|\rho^{\nu}(s)\|_{L^{\infty}}\right) \, ds.
\end{equation*}

To complete the proof of \cref{T:HolderUniform}, we apply Proposition 2.3.5 of \cite{C1998} and write 
\begin{equation*}
\begin{split}
&\| \nabla \v^{\nu}(t)\|_{L^{\infty}} \leq \frac{C}{\beta} \| \nabla \v^{\nu}(t)\|_{C^0_*} \log \left( e + \frac{\| \nabla \v^{\nu}(t)\|_{C^{\beta}}}{\| \nabla \v^{\nu}(t)\|_{C^0_*}}  \right).
\end{split}
\end{equation*}
Since $x \mapsto x\log\left( e + \frac{C}{x}  \right)$ is increasing in $x$ when $C>0$, it follows from \cref{L:VelocityReg}, the embedding $L^{\infty}\hookrightarrow C^0_*$, and the equivalence between $C^{\beta}$ and $C^{\beta}_*$ when $\beta\in(0,1)$ that
\begin{equation*}
\begin{split}
&\| \nabla \v^{\nu}(t)\|_{L^{\infty}} \leq \frac{C}{\beta} \|\rho^{\nu} \|_{L^1\cap L^{\infty}} \log \left( e + \frac{\| \nabla \v^{\nu}(t)\|_{C^{\beta}}}{\|\rho^{\nu} \|_{L^1\cap L^{\infty}}}  \right) \leq \frac{C}{\beta} \|\rho^{\nu} \|_{L^1\cap L^{\infty}} \log \left( e + \frac{\| \rho^{\nu}(t)\|_{L^1\cap C^{\beta}}}{\|\rho^{\nu} \|_{L^1\cap L^{\infty}}}  \right)\\
&\qquad \leq \frac{C}{\beta} C_0(T) \log \left( e + \frac{\| \rho^{\nu}(t)\|_{L^1\cap C^{\beta}}}{C_0(T)}  \right),
\end{split}
\end{equation*}
where we used \cref{T:ViscousExistence} in the third inequality above.  An application of \ref{mainholderboundest} gives
\begin{equation*}
\begin{split}
&\| \nabla \v^{\nu}(t)\|_{L^{\infty}} \leq C_0(T) \log \left( e + Ce^{CV(t)} \| \rho_0\|_{C^{\beta}}  \right)\\
&\qquad \leq C_0(T) \log \left( e^{CV(t)}\left(e + \| \rho_0\|_{C^{\beta}} \right) \right) = C_0(T)\left( V(t)+ \log\left(e +  \| \rho_0\|_{C^{\beta}}  \right)\right)\\
&\qquad = C_0(T) \log\left(e +  \| \rho_0\|_{C^{\beta}}  \right) + C_0(T)\int_0^t \left(\| \nabla \v^{\nu} (s)\|_{L^{\infty}}+ \|\rho^{\nu}(s)\|_{L^{\infty}}\right) \, ds.
\end{split}
\end{equation*}
Gronwall's lemma and \cref{T:ViscousExistence} imply that 
\begin{equation*}
\| \nabla \v^{\nu}(t)\|_{L^{\infty}} \leq C_0(T)e^{C_0(T)}.
\end{equation*}
Substituting this estimate into (\ref{mainholderboundest}) gives (\ref{holderuniform}).
\end{proof}

We now prove the main result of this section.
\begin{theorem}\label{T:VVSUP}
Assume $\rho_0$ is compactly supported and belongs to $C^{\alpha}(\R^d)$ with $\alpha>1$, $d\geq 2$.  Define $\mu$ as in \cref{e:muwDefs}, and let $T$ be as in \cref{T:ViscousExistence}.  Then for $t\in[0,T]$, $\nu\leq 1$, and $\beta\in(0,1)$,
\begin{equation*}
	\| \mu(t) \|_{L^{\infty}}
		\leq C_0(t) (\nu t)^{\frac{2\beta}{2\beta+d}}.
\end{equation*}
\end{theorem}
\begin{proof}
Fix $t<T$, $p>d$, and  $\beta\in(0,1)$.  For fixed $N\geq 0$ (to be chosen later), we use Bernstein's Lemma and the definition of the Holder space $C^{\beta}(\R^d)$ as given in \cref{D:HolderSpaces} to write
\begin{equation}\label{LPest}
\begin{split}
&\| (\rho^{\nu}  - \rho^0)(t) \|_{L^{\infty}}  \leq \sum_{j=-1}^N \| \Delta_j (\rho^{\nu}  - \rho^0)(t) \|_{L^{\infty}} + \sum_{j={N+1}}^{\infty} \| \Delta_j (\rho^{\nu}  - \rho^0)(t) \|_{L^{\infty}}\\
&\qquad \leq C\sum_{j=-1}^N 2^{j\frac{d}{2}} \| \Delta_j (\rho^{\nu}  - \rho^0)(t) \|_{L^{2}} + \sum_{j={N+1}}^{\infty} 2^{-j\beta}2^{j\beta}\| \Delta_j (\rho^{\nu}  - \rho^0)(t) \|_{L^{\infty}}\\
&\qquad \leq C\sum_{j=-1}^N 2^{j\frac{d}{2}} \| \Delta_j (\rho^{\nu}  - \rho^0)(t) \|_{L^{2}} + C\sum_{j={N+1}}^{\infty} 2^{-j\beta}(\| \rho^{\nu}(t)\|_{C^{\beta}} + \|\rho^{0}(t) \|_{C^{\beta}} )\\
&\qquad \leq C2^{N\frac{d}{2}} \| (\rho^{\nu}  - \rho^0)(t) \|_{L^{2}} + C_0(t)2^{-N\beta},
\end{split}
\end{equation}
where we applied \cref{T:HolderUniform,T:InviscidExistence} above to get the last inequality.  By Theorems \ref{T:VVPAG} and \ref{T:VVGAG}, for $\nu\leq 1$, 
\begin{equation*}
\| (\rho^{\nu}  - \rho^0)(t) \|_{L^{2}} \leq C_0(t)t\nu e^{C_0(t)}.
\end{equation*}
Substituting this estimate into (\ref{LPest}) gives
\begin{equation*}
\| (\rho^{\nu}  - \rho^0)(t) \|_{L^{\infty}}  \leq C_0(t)te^{C_0(t)} \nu  2^{N\frac{d}{2}}  + C_0(t)2^{-N\beta}.
\end{equation*}
Now set $N=-{\frac{2}{2\beta+d}}\log_2 (\nu t)$. We conclude that 
\begin{equation*}
\begin{split}
	&\| (\rho^{\nu}  - \rho^0)(t) \|_{L^{\infty}}
		\leq C_0(t)te^{C_0(t)} (\nu t)^{1-\frac{d}{2\beta+d}}
			+ C_0(t) (\nu t)^{\frac{2\beta}{2\alpha+d}}
		\leq C_0(t)e^{C_0(t)} (\nu t)^{\frac{2\beta}{2\beta+d}}.
\end{split}
\end{equation*}  
This completes the proof of \cref{T:VVSUP}.
\end{proof}

Above, we used the following lemma.

\begin{lemma}\label{L:VelocityReg}
	For all $r \in \R$,
	\begin{align*}
		\norm{\grad\grad \F * \rho}_{C_*^{r}}
			\le C (\norm{\rho}_{L^1} +  \norm{\rho}_{C_*^{r}}).
	\end{align*}
\end{lemma}
\begin{proof}
	Let $\v = \grad \F * \rho$. Then
	using the
	definition of $C^r_*$ as  given
	in \cref{D:HolderSpaces}, we have
	\begin{align*}
		\norm{ \grad \v}_{C_*^r}
			&= \sup_{q \ge -1} 2^{qr}\norm{\Delta_q \grad \v}_{L^\iny}
			\le 2^{-r} \norm{\Delta_{-1} \grad \v}_{L^\iny}
				+  \sup_{q \ge 0} 2^{ q r} \norm{\Delta_q  \grad \v}_{L^\iny} \\
			&\le C \norm{ \Delta_{-1}\rho}_{L^1}
				+ C \sup_{q \ge 0} 2^{ qr} \norm{\Delta_q \rho}_{L^\iny}
			\le C \norm{\rho}_{L^1}
				+ C \norm{\rho}_{C_*^{r}}.
	\end{align*}
	To obtain the second inequality above, we argued as in (3.6) of \cite{VishikBesov} when estimating the low frequency term and applied a classical lemma to bound the high frequencies (see, for example, Lemma 4.2 of \cite{CK2006}).  This completes the proof.
\end{proof}


%
%
\section{Concluding Remarks}

\noindent It is possible to obtain a velocity formulation of \GAGnu, in analogy with the Navier-Stokes and Euler equations. For any $\nu \ge 0$, we can write this in the form,
\begin{align*}
	\left\{
		\begin{array}{l}
			\prt_t \v^\nu + \v^\nu \cdot \grad \v^\nu + \grad q^\nu = \nu \Delta \v^\nu, \\
			\curl \v^\nu = 0, \\
			\v^\nu(0) = \v_0,
		\end{array}
	\right.
\end{align*}
for an appropriately chosen ``pressure,'' $q^\nu$. This velocity formulation can be used to obtain the bounds on $\norm{\w(t)}$ in \cref{S:VV} and $\norm{\wl(t)}$ in \cref{S:VVNonL2}. Because $\dv \w \ne 0$, however, the pressure does not disappear in these bounds. This requires a great deal of effort to properly bound the pressure, so we took the shorter approach in \cref{S:VV,S:VVNonL2}, leaving the elaboration of the velocity formulation to future work.

\section*{Acknowledgements}
\noindent Work on this paper was partially supported by NSF grant DMS-1212141.

\appendix




\def\cprime{$'$} \def\polhk#1{\setbox0=\hbox{#1}{\ooalign{\hidewidth
  \lower1.5ex\hbox{`}\hidewidth\crcr\unhbox0}}}

\end{document}